\documentclass{amsart}

\title[Rigorous enclosure of Lyapunov exponent for Gaussian R.D.S.]{ Efficient computation of stationary measures and the Lyapunov Landscape for families of random dynamical systems with smooth additive noise }

\author{S. Galatolo, C. Lopez Vereau, L. Marangio, I. Nisoli}

\newtheorem{definition}{Definition}[section]
\newtheorem{theorem}[definition]{Theorem}
\newtheorem{proposition}[definition]{Proposition}
\newtheorem{lemma}[definition]{Lemma}
\newtheorem{corollary}[definition]{Corollary}
\newtheorem{remark}[definition]{Remark}

\newtheorem*{main_application}{Main result 2}
\newtheorem*{main_computability}{Main result 1}

\usepackage{graphicx}
\usepackage{xcolor}

\usepackage{hyperref}
\usepackage{placeins}
\usepackage{subcaption}

\usepackage{lscape}

\begin{document}

\maketitle

\section*{Abstract}

We present an efficient and validated method for approximating the stationary measures of random dynamical systems with smooth additive noise. The approach leverages the strong regularizing properties of the associated transfer operator through a finite-dimensional reduction based on Fourier approximation. Explicit error bounds make the method suitable for use in computer-assisted proofs and rigorous numerical investigations; in particular, its efficiency {\em enables systematic exploration of parameter space}.

The method provides access to the stationary measure and supports the analysis of key statistical properties of the system. As an application, we study noise-induced phenomena, focusing on the transition from positive to negative Lyapunov exponent—one of the telltale signs of a physical phenomenon known as \emph{Noise Induced Order}—in families of random unimodal maps with Gaussian additive noise.

By analyzing the Lyapunov exponent as a function of the system parameters, we identify transitions along a hypersurface in parameter space. The parameters we consider include the standard deviation (intensity) of the Gaussian noise and the shape of the unimodal map.

\section{Introduction}

In random dynamical systems, \emph{Noise Induced Order} (NIO) refers to a set of phenomena in which increasing the noise intensity can lead to the emergence of more regular dynamical behavior. Typical manifestations include:
\begin{itemize}
	\item sharpening of the power spectrum,
	\item abrupt decrease in entropy,
	\item appearance of a negative Lyapunov exponent,
	\item spatial localization of orbits.
\end{itemize}
In this paper we focus on the transition from positive to negative Lyapunov exponent as the noise intensity increases. This particular signature of NIO was first observed numerically in 1983 in the Matsumoto--Tsuda model for the Belousov--Zhabotinsky chemical reaction~\cite{MT83}. It illustrates a striking effect: the system appears chaotic for very small noise intensity, but as the noise increases, it exhibits increasingly regular behavior, eventually crossing into a regime with negative Lyapunov exponent. Similar qualitative transitions have been reported for other indicators of chaos in both low- and high-dimensional systems.

The numerical observations of Matsumoto--Tsuda \cite{MT83} triggered a large interdisciplinary literature (over 450 citations at the time of writing), with extensive numerical and applied investigations of noise-driven stabilization and related transitions.
At the same time, the subset of this literature that provides \emph{mathematical} support directly for the Matsumoto--Tsuda noise-induced order mechanism (as reflected, for instance, in MathSciNet-indexed work) is comparatively small.
This point is explicitly noted by Sumi, who observes that many physicists investigated the phenomenon numerically, ``although there has been only a few mathematical supports for it'' \cite{Sumi2011}.

In our context, by ``rigorous proof'' we mean a validated chaos-to-order transition for the original Matsumoto--Tsuda setting, in particular a certified sign change of the Lyapunov exponent with explicit error bounds.
Such a computer-assisted result appeared only recently \cite{G2020}.

The comparatively recent appearance of rigorous proofs is not merely historical: the original Matsumoto--Tsuda and Lasota--Mackey models pose genuine obstacles to a purely \emph{a priori} analysis. They fall outside the most convenient uniformly expanding / bounded-distortion paradigms, and the noise-induced order mechanism is tied to parameter regimes where non-uniform behavior and limited regularity prevent straightforward \emph{a priori} control.
The results in \cite{G2020,S2022} become possible via an \emph{a posteriori}, computer-assisted approach.
The most delicate step is to prove, with certification, that an appropriate finite-rank discretization has a spectral gap (equivalently, explicit mixing bounds on the zero-average subspace).
Coarse--fine inequalities then allow one to lift these discrete estimates to explicit bounds for the annealed transfer operator itself, yielding rigorous error control for the stationary density and the associated Lyapunov exponent. Our approach follows this philosophy, providing a validated Fourier--Galerkin framework for smooth Gaussian noise with explicit constants enabling systematic, error-controlled exploration of parameter space.

Further validated developments include multiple sign transitions in the Lasota--Mackey map \cite{S2022}, related behavior for unimodal maps with bounded-variation noise \cite{Nisoli2022}, and extensions to random skew-products over non-uniformly expanding bases \cite{BlumenthalNisoli2022}.
More recently, analogous phenomena have been investigated in high-dimensional systems with structured stochastic dynamics \cite{chen2024noiseinducedorderhighdimensions}. Alongside Noise Induced Order, other noise-induced phenomena such as \emph{Noise Induced Chaos} have been studied in the applied literature \cite{Sato2009}. These phenomena are strongly tied to the statistical properties of typical trajectories--e.g., the sign of the Lyapunov exponent in one-dimensional maps depends on whether expansion or contraction dominates along typical orbits. These statistical properties, in turn, are governed by the system's stationary measures (as we discuss in Section~\ref{sec:randskew}), which can be characterized as fixed points of suitable transfer operators (see Section~\ref{c21}). The presence of noise has a regularizing effect on the transfer operators, making them and their fixed points approximable via finite-element methods: the core of the computer-assisted strategies we use. In this sense, noise helps the computer capture the dynamics at finite resolution in a way that enables rigorous inference about the underlying system.

Computer-assisted approaches to the approximation of invariant measures via discretization of the transfer operator have been extensively developed in the deterministic case, beginning with Ulam's pioneering ideas \cite{Ulam}. See, e.g., \cite{BB, BS, BM, FC, DelJu02, GMNP, GN, GaNi2, KMY, L, W} for developments using various approximation schemes, including Fourier and polynomial bases. Other strategies relying on periodic orbits have also been explored \cite{PJ99}.

The strategy we apply here to get explicit estimates on the finite elements reduction error,  relies on the {\em a posteriori approach} proposed in \cite{GN}, allowing the convergence to equilibrium to be tested on the discretized transfer operator. With this approach the convergence to equilibrium estimate is performed by the computer instead of being given a priori by general considerations on the system.
To allow such a complicated estimate to be performed efficiently we also implement a multi resolution approach, we   called "coarse-fine strategy" (\cite{GMNP},\cite{GNS}) where the convergence to equilibrium of a large finite state markov chain is estimated by the convergence to equilibrium of a lower resolution version of it satisfying the same Lasota-Yorke-Doeblin-Fortet general estimates. 

From an abstract computability perspective, \cite{GHR} shows that even some computable deterministic systems may lack computable invariant measures. Similar noncomputability results are known for physical measures of smooth maps \cite{RY}. However, such barriers are removed in the presence of random noise: \cite{BGR} shows that in a compact phace space,  random systems with additive noise supported on a small ball always admit computable stationary measures.

Nevertheless, efficient algorithms for approximating stationary measures in the random setting remain relatively scarce \cite{G2020,MSDGG,Froyland_2014}. This paper aims to provide a highly efficient solution to this problem when the noise is Gaussian. While results are available for stochastic differential equations (see \cite{BredenEngel2023}), and in the SDE case alternatives such as adjoint methods exist \cite{breden2025rigorousenclosurelyapunovexponents}, our approach is based directly on discretization of the transfer operator using a Fourier scheme.

We demonstrate the effectiveness of the method with high-accuracy enclosures of the stationary measure and Lyapunov exponent in concrete one-dimensional examples. We then apply these tools to study transitions such as Noise Induced Order and Noise Induced Chaos.

We thank the referee for pointing out the close conceptual link with~\cite{G2020}.
While the overarching philosophy is similar---validated discretization of the annealed transfer operator combined with a posteriori error control---the present work operates in a different regularity regime and therefore relies on a substantially different (and complementary) numerical strategy.

In~\cite{G2020} the noise has bounded variation and compact support, providing only BV-type regularization. 
Here the additive noise is Gaussian, hence smooth with an analytic kernel; this yields much stronger smoothing, in particular analyticity of the stationary density and rapid (exponentially fast) decay of Fourier coefficients. 
This structural difference is the main reason why the present method attains high precision with a comparatively small truncation dimension.

Moreover, the BV/Ulam-type discretization in~\cite{G2020} produces very large (though sparse) Markov matrices and typically requires fine partitions (and thus heavier computations) to obtain tight enclosures; we refer to~\cite{GMNP} for an in-depth analysis of the associated coarse-fine certification strategy. 
In contrast, we employ a Fourier--Galerkin scheme tailored to Gaussian smoothing: the noise operator is diagonal in the Fourier basis and high modes are exponentially suppressed. 
As a consequence, although the resulting matrices are dense, they are of much smaller dimension for comparable (indeed higher) accuracy, leading to a significantly faster validated computation.

Finally, this efficiency makes it practical here to rigorously explore broad regions of parameter space (producing validated ``Lyapunov landscapes'' and locating zero-crossing hypersurfaces), and to reach smaller noise intensities while maintaining rigorous error control.
The approach of~\cite{G2020} is instead well-suited for targeted investigations (e.g.\ slices or isolated parameter values) in settings where a BV framework is natural (as in many applied models) and the sparse structure is advantageous. 
In this sense, the two methods are complementary:~\cite{G2020} provides computer-assisted proofs of NIO in the BV-noise setting, whereas the present paper develops a high-precision Fourier framework exploiting Gaussian regularization to enable systematic rigorous exploration across parameter space.

\section{Summary of the Results}

This paper has two main contributions:

\begin{itemize}
\item We develop a general theory for approximating the stationary measure of random dynamical systems with Gaussian noise in \( L^2([-1,1]) \), including rigorous error bounds based on a Fourier finite-element discretization of the transfer operator. This framework allows one to rigorously enclose the Lyapunov exponent.

\item We apply this method to study a family of unimodal maps previously introduced in \cite{Nisoli2022}, demonstrating the presence of multiple noise-induced transitions and visualizing the landscape of the Lyapunov exponent in parameter space.
\end{itemize}

The following computability result underpins the general theory:

\begin{main_computability}
Let \( T : [-1,1] \to \mathbb{R} \) be a measurable map such that the pushforward of Lebesgue measure by \( T \) is absolutely continuous with respect to Lebesgue measure (i.e., \( T \) is non-singular). Define the random dynamical system with Gaussian noise:
\[
X_{n+1} = \tau(T(X_n) + \Omega_n),
\]
where \( \Omega_n \sim \mathcal{N}(0,\sigma^2) \) are i.i.d. random variables, and \( \tau : \mathbb{R} \to [-1,1] \) is a periodic boundary condition.

Then, for every \( \sigma > 0 \), the system admits a unique stationary measure \( \mu_\sigma = f_{\sigma} m \), with analytic density \( f_\sigma\), where \( m \) is the Lebesgue measure on \( [-1,1] \).

Moreover, for every observable \( \phi \in L^2(m) \), the Birkhoff average
\[
\frac{1}{N} \sum_{n=0}^{N-1} \phi(X_n)
\]
converges almost surely to \( \int \phi \, d\mu_\sigma \), and this integral can be rigorously enclosed with arbitrary precision using the algorithm presented in Section \ref{c23}.
\end{main_computability}

The proof of this result is developed throughout the paper. Existence of the stationary measure is established in Lemma~\ref{lemma:existence}, while uniqueness is shown in Proposition~\ref{prop:unique} as a consequence of the Doeblin argument used to estimate mixing rates in Proposition~\ref{prop:Doeblin}. Analyticity of the stationary density is proved in Corollary~\ref{cor:analytic}. Furthermore, the same Doeblin-type estimates yield the computability results discussed in Section~\ref{sec:computability}, namely, the computability of the stationary measure in Theorem~\ref{thm:computability} and of Birkhoff averages in Corollary~\ref{cor:BirkhoffAvg}.

As an application, we consider a discrete-time random dynamical system defined by the two-parameter family of maps
\begin{equation}\label{mappa}
T_{\alpha,\beta}(x) = \beta - (1 + \beta)|x|^\alpha
\end{equation}
with Gaussian additive noise. Some sufficient conditions for a transition of Noise Induced Order type were studied in \cite{Nisoli2022} in some case of bounded variation noise. Here in the  Gaussian case we provide a more refined  description of both Noise Induced Order and Noise Induced Chaos exploring a whole region of the parameter space.

We rigorously enclose regions of parameter space where the Lyapunov exponent transitions from positive to negative, proving the existence of multiple noise-induced transitions.

This leads to the following computer-assisted result:

\begin{main_application}\label{teo:application}
Let \( T_{\alpha,\beta} : [-1,1] \to \mathbb{R} \) be as above, and let \( \tau : \mathbb{R} \to [-1,1] \) denote the projection modulo 2. Consider the random system
\[
X_{n+1} = \tau(T_{\alpha,\beta}(X_n) + \Omega_n),
\]
where \( \Omega_n \sim \mathcal{N}(0,\sigma^2) \) are i.i.d. Gaussian variables.

Then the top Lyapunov exponent \( \lambda(\alpha, \beta, \sigma) \) is well-defined for all \( \alpha \in [1, \infty), \beta \in (-1,1], \sigma > 0 \). Moreover:
\begin{itemize}
  \item For \( \beta = 1 \), the function \( \alpha \mapsto \lambda(\alpha,1,\sigma) \) changes sign as \( \sigma \) increases, for \( \alpha \in [3,4] \), from positive to negative.
  \item For \( \alpha = 3 \), the function \( \beta \mapsto \lambda(3,\beta,\sigma) \) exhibits multiple sign changes as \( \sigma \) increases, for \( \beta \in [0.8447, 0.8694] \).
\end{itemize}
\end{main_application}

\begin{remark}
The proof of Main Result 2 is obtained by rigorously enclosing the Lyapunov exponent using validated numerics.

Specifically:
\begin{itemize}
  \item For \( \beta = 1 \), the Lyapunov exponent was enclosed rigorously at \( \alpha = 3 + \frac{i}{1024} \), for \( i = 0, \ldots, 1024\), $\sigma = 1/16+ 15/16 \frac{j}{1024}$ for \( j = 0, \ldots, 1024 \).
  \item For \( \alpha = 3 \), the Lyapunov exponent was enclosed rigorously at \( \beta = \frac{51}{64} + \frac{i}{8192} \), for \( i = 1, \ldots, 1024 \), $\sigma = 1/16+ 15/16 \frac{j}{1024}$ for \( j = 0, \ldots, 1024 \).
\end{itemize} 
\end{remark}





It is known that in the present case, considering smooth families of  systems with additive smooth noise and periodic boundary conditions (\cite{GS}) the stationary measure and hence the Lyapunov exponent vary in a smooth way as the parameters vary.
The following figures show in the case of the family \ref{mappa} with Gaussian noise, how the sign of the Lyapunov exponent varies with the parameters, by rigorously computed enclosures produced by our code:
\begin{itemize}
	\item Figure~\ref{NoiseInducedOrder} shows the regions in the \((\alpha, \sigma)\)-plane where the Lyapunov exponent transitions from positive (red) to negative (blue).
	\item Figure~\ref{NoiseInducedChaos2} shows the analogous transition in the \((\beta, \sigma)\)-plane for fixed \(\alpha = 3\).
\end{itemize}

\begin{figure}[!htb]
	\centering
	\includegraphics[width=80mm]{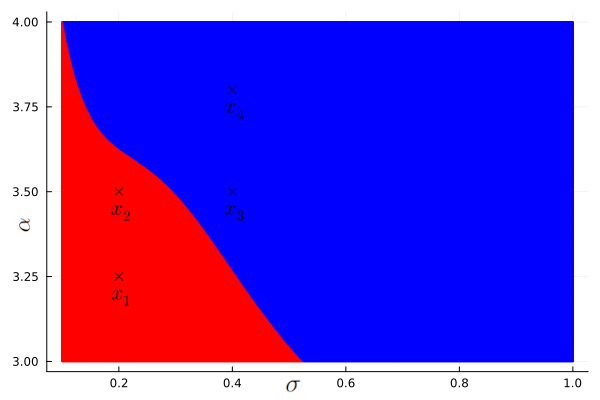}
	\caption{Parameter region in the \((\alpha, \sigma)\)-plane. Blue: \(\lambda < 0\), Red: \(\lambda > 0\).
	Rigorous enclosures for the Lyapunov exponent at the points marked with black crosses can be found in Table \ref{tab:valNIO}.
	\label{NoiseInducedOrder}
	The Lyapunov exponent was rigorously enclosed on the points of the grid \( \alpha = 3 + \frac{i}{1024} \), for \( i = 0, \ldots, 1024\), $\sigma = 1/16+ 15/16 \frac{j}{1024}$ for \( j = 0, \ldots, 1024 \)}. 
\end{figure}

\begin{figure}[!htb]
	\centering
	\includegraphics[width=80mm]{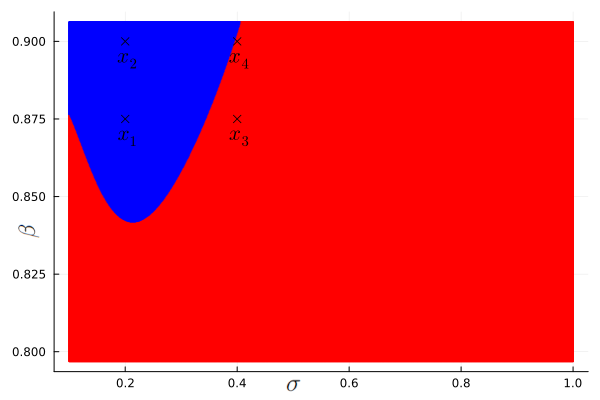}
	\caption{Parameter region in the \((\beta, \sigma)\)-plane for \(\alpha = 3\). Blue: \(\lambda < 0\), Red: \(\lambda > 0\).
	Rigorous enclosures for the Lyapunov exponent at the points marked with black crosses can be found in Table \ref{tab:valNIC}.
	The Lyapunov exponent was enclosed rigorously on the points of the grid \( \beta = \frac{51}{64} + \frac{i}{8192} \), for \( i = 1, \ldots, 1024 \), $\sigma = 1/16+ 15/16 \frac{j}{1024}$ for \( j = 0, \ldots, 1024 \).}
	\label{NoiseInducedChaos2}
\end{figure}

We emphasize that both Figure~\ref{NoiseInducedOrder} and Figure~\ref{NoiseInducedChaos2} are obtained using fully rigorous numerical methods. As such, up to errors introduced by graphical rendering (e.g., rasterization or color interpolation), they can be interpreted as certified results. In particular, they provide an effective visualization of \emph{Main Result 2} stated above. The crosses in the figures correspond to sample enclosures of the Lyapunov exponent presented in Tables~\ref{tab:valNIO} and~\ref{tab:valNIC}, we report as examples of the computation done, presenting explicit values and enclosure bounds; the data is available at \cite{LyapData}.

The values reported in Table~\ref{tab:valNIO} were computed to machine precision using a Galerkin truncation at frequency $128$, with a total runtime of $2.02$ seconds on a single thread of a Ryzen 5 5600 processor. No parallelization was used, and the annealed transfer operator was recomputed independently for each pair of parameter values. Consequently, the average runtime for the full computation -including operator discretization, mixing time estimation, and rigorous Lyapunov exponent evaluation- is approximately $0.5$ seconds per parameter point. Despite its simplicity, this naive implementation already delivers high performance, achieving $15$ certified digits without requiring any specialized optimization.

For the full parameter space plots (Figures~\ref{NoiseInducedOrder} and~\ref{NoiseInducedChaos2}), we employed only minimal optimization: in particular, we computed the transfer operator of the deterministic map once for each fixed value of $\alpha$ or $\beta$, and reused it across all $\sigma$ values. This reduced computational cost while preserving the rigor of the enclosures. As a result, our algorithm remains both validated and efficient, and is well suited for the systematic exploration of large parameter regions in a robust way.

We stress again that both the algorithm and the discretization scheme are independent of the specific dynamical properties of the underlying deterministic map. This reflects an instance of the a posteriori validation philosophy, which in our context evolves from the coarse-fine strategy developed in~\cite{GMNP}. The essential analytical estimates—such as mixing and spectral contraction—are established through rigorous linear algebra computations performed on the discretized operator. As a result, our method is problem-agnostic within the class of random dynamical systems with Gaussian noise superimposed on non-singular interval maps. This structural independence underlines the generality and flexibility of our approach.

\begin{table}[h!]
\centering
\begin{tabular}{cccc}
Label &$\alpha$ & $\sigma$ & $\lambda$ \\
$x_1$ &$3.25$ & $0.2$ & $0.139610862369467 \pm 10^{-15}$\\
$x_2$ &$3.5$ & $0.2$ & $0.047682027067898 \pm 10^{-15}$ \\
$x_3$ &$3.5$ & $0.4$ & $-0.095573727164159 \pm 10^{-15}$ \\
$x_4$ &$3.8$ & $0.4$ & $-0.223066357002470 \pm 10^{-15}$ \\
\end{tabular}
\caption{Sample values in Figure \ref{NoiseInducedOrder}\label{tab:valNIO}}
\end{table}

\begin{table}[h!]
\centering
\begin{tabular}{cccc}
Label & $\beta$ & $\sigma$ & $\lambda$ \\
$x_1$ & $0.875$ & $0.2$ & $0.058040752469902 \pm 10^{-15}$\\
$x_2$ & $0.9$ & $0.2$ & $0.098888825866413\pm 10^{-15}$ \\
$x_3$ & $0.875$ & $0.4$ & $-0.036288212585984 \pm 10^{-15}$ \\
$x_4$ & $0.9$ & $0.4$ & $-0.003180843719330 \pm 10^{-15}$ \\
\end{tabular}
\caption{Sample values in Figure \ref{NoiseInducedChaos2}\label{tab:valNIC}}
\end{table}

It is worth emphasizing that linear response methods have been developed and adapted to the class of perturbations studied in this paper \cite{GG, GS}. In particular, the theory of linear response for systems perturbed by Gaussian noise, as developed in \cite{GS}, combined with the implicit function theorem, supports our claim that the zeros of the Lyapunov exponent lie along a hypersurface in parameter space. A more detailed investigation of the dynamics associated with parameters on this hypersurface presents a promising direction for future research, as such systems may exhibit intermittent-like behavior.

Indeed, recent work by Goverse and collaborators \cite{goverse2025intermittenttwopointdynamicstransition} shows that, at the onset of chaos (i.e., when the Lyapunov exponent vanishes) in random circle maps, one observes intermittent two-point dynamics and infinite invariant measures. Intermittent two point behavior is exactly the behavior that may arise in our setting along the hypersurface where the Lyapunov exponent is zero. Numerical evidence of similar intermittent two-point dynamics has also been observed in our simulations; see Figure~\ref{fig:two_point}. In Table~\ref{tab:crossing}, we report validated enclosures for values of the noise amplitude $\sigma$ as a function of the shape parameter $\alpha$ at which this transition occurs, and in Figure~\ref{fig:transitions}, we display an enclosure of the function $\sigma(\alpha)$ at sample points. We stress, however, that our current computations only establish the existence of such transitions, without addressing their uniqueness or potential multiplicity.

\begin{figure}[htbp]
  \centering
  \begin{subfigure}{0.33\textwidth}
    {\includegraphics[width=\linewidth]{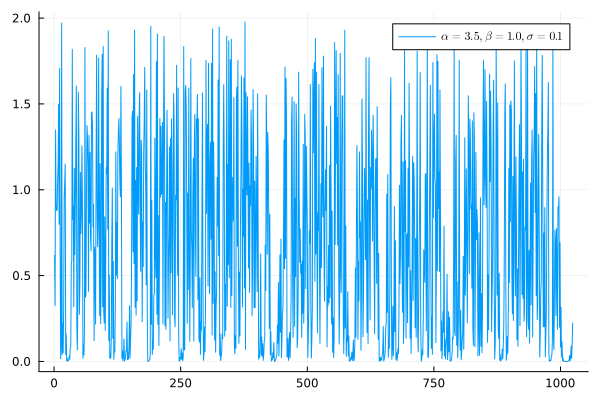}}
  \end{subfigure}
  \begin{subfigure}{0.33\textwidth}
    {\includegraphics[width=\linewidth]{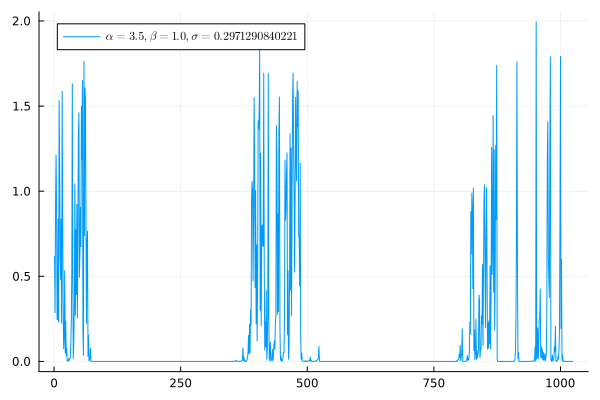}}
  \end{subfigure}
  \begin{subfigure}{0.33\textwidth}
    {\includegraphics[width=\linewidth]{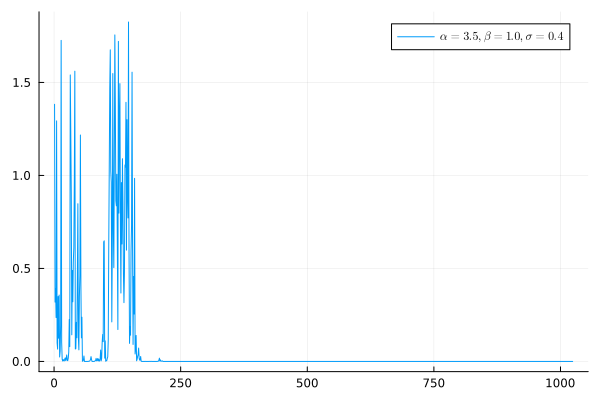}}
  \end{subfigure}
  \caption{Numerical experiments plotting $|T^i_{\omega}x_0-T^i_{\omega}(y_0)|$ for initial conditions $x_0 = -0.1, y_0 = 0.9$
  for noise realizations with $\sigma = 0.1, 0.2971290840221, 0.4$; the transition from positive to negative Lyapunov exponent 
  was rigorously enclosed in $[0.2971290840221, 0.2971290840222]$.}
  \label{fig:two_point}
\end{figure}

\begin{figure}[!htb]
	\centering
	\includegraphics[width=80mm]{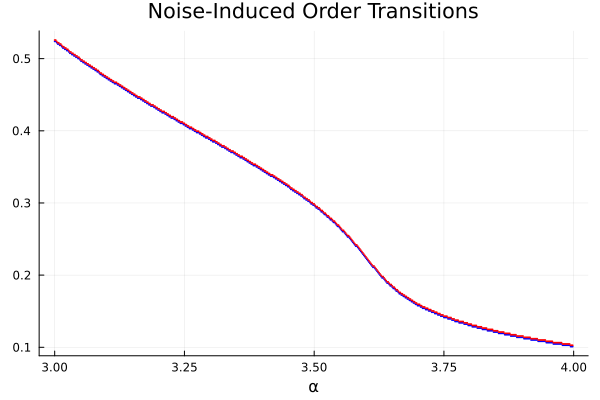}
	\caption{Enclosure of the transition value $\sigma$ as a function of $\alpha$; in blue the lower bound, in red the upper bound.
	This figure gives a graphical representation of the definition at which the hypersurface of zero Lyapunov exponent is known in our rigorous data.}
	\label{fig:transitions}
\end{figure}

\begin{landscape}
	\begin{table}[ht]
		\centering
		\caption{Sample of detected Lyapunov Crossings}\label{tab:crossing}
		\begin{tabular}{|c|c|c|c||c|c|}
			\hline
			\textbf{Row} &$\alpha$ & $\sigma_1$ & $\lambda_1$  & $\sigma_2$ & $\lambda_2$ \\
			\hline
$1$ & $3.0$ & $0.523926$ & $[0.00107101, 0.00107102]_{256}$ & $0.525757$ & $[-0.0005515, -0.000551499]_{256}$ \\
$2$ & $3.04883$ & $0.498291$ & $[0.000820052, 0.000820053]_{256}$ & $0.500122$ & $[-0.000905828, -0.000905827]_{256}$ \\
$3$ & $3.09766$ & $0.474487$ & $[0.000816379, 0.00081638]_{256}$ & $0.476318$ & $[-0.000986452, -0.000986451]_{256}$ \\
$4$ & $3.14648$ & $0.452515$ & $[0.000518536, 0.000518537]_{256}$ & $0.454346$ & $[-0.00133441, -0.0013344]_{256}$ \\
$5$ & $3.19531$ & $0.430542$ & $[0.00132734, 0.00132735]_{256}$ & $0.432373$ & $[-0.000545068, -0.000545067]_{256}$ \\
$6$ & $3.24414$ & $0.4104$ & $[0.000980779, 0.00098078]_{256}$ & $0.412231$ & $[-0.000883612, -0.000883611]_{256}$ \\
$7$ & $3.29297$ & $0.390259$ & $[0.000958772, 0.000958773]_{256}$ & $0.39209$ & $[-0.000864106, -0.000864105]_{256}$ \\
$8$ & $3.3418$ & $0.370117$ & $[0.000848431, 0.000848432]_{256}$ & $0.371948$ & $[-0.000897693, -0.000897692]_{256}$ \\
$9$ & $3.39062$ & $0.349976$ & $[0.000214735, 0.000214736]_{256}$ & $0.351807$ & $[-0.00141992, -0.00141991]_{256}$ \\
$10$ & $3.43945$ & $0.328003$ & $[0.000105052, 0.000105053]_{256}$ & $0.329834$ & $[-0.00136926, -0.00136925]_{256}$ \\
$11$ & $3.48828$ & $0.302368$ & $[0.000835062, 0.000835063]_{256}$ & $0.304199$ & $[-0.000416835, -0.000416834]_{256}$ \\
$12$ & $3.53711$ & $0.273071$ & $[0.000731332, 0.000731333]_{256}$ & $0.274902$ & $[-0.000262686, -0.000262685]_{256}$ \\
$13$ & $3.58594$ & $0.23645$ & $[0.00018605, 0.000186051]_{256}$ & $0.238281$ & $[-0.000581865, -0.000581864]_{256}$ \\
$14$ & $3.63477$ & $0.194336$ & $[0.000300537, 0.000300538]_{256}$ & $0.196167$ & $[-0.000607178, -0.000607177]_{256}$ \\
$15$ & $3.68359$ & $0.165039$ & $[0.000930034, 0.000930035]_{256}$ & $0.16687$ & $[-0.000639231, -0.00063923]_{256}$ \\
$16$ & $3.73242$ & $0.146729$ & $[0.00171531, 0.00171532]_{256}$ & $0.14856$ & $[-0.000694382, -0.000694381]_{256}$ \\
$17$ & $3.78125$ & $0.133911$ & $[0.00236626, 0.00236627]_{256}$ & $0.135742$ & $[-0.000913736, -0.000913735]_{256}$ \\
$18$ & $3.83008$ & $0.124756$ & $[0.00134207, 0.00134208]_{256}$ & $0.126587$ & $[-0.00273638, -0.00273637]_{256}$ \\
$19$ & $3.87891$ & $0.115601$ & $[0.00469925, 0.00469926]_{256}$ & $0.117432$ & $[-0.000340576, -0.000340575]_{256}$ \\
$20$ & $3.92773$ & $0.110107$ & $[0.00127257, 0.00127258]_{256}$ & $0.111938$ & $[-0.00444233, -0.00444232]_{256}$ \\
$21$ & $3.97656$ & $0.104614$ & $[1.11679\times10^{-5}, 1.1168\times10^{-5}]_{256}$ & $0.106445$ & $[-0.00644444, -0.00644443]_{256}$ \\
			\hline
		\end{tabular}
	\end{table}
\end{landscape}

The full results of the computation are accessible at Harvard Dataverse \cite{LyapData}; the data is provided both as a dataframe 
in JLD2 format and in CSV format.
The $l^2$ mixing rates for the iterations of the discretized operator are provided, which can be used to estimate the mixing rates
for the abstract annealed operator or as an ingredient in further rigorous computations, as those of linear response and 
diffusion coefficients (see \cite{B2016} and \cite{B2018} where it is shown how the validated computation of the mixing rate allows the validated  computation of diffusion coefficients and linear response).
In figure \ref{fig:mixing} we present heatmaps built upon the rigorous mixing bounds in the $\alpha$, $\sigma$ plane when $\beta = 1.0$. Here, $\mathcal{U}_0 \subset L^2([-1,1])$ denotes the subspace of average $0$ functions in $L^2$; its formal definition is given later in Definition~\ref{def:average0}.

\begin{figure}[htbp]
  \centering
  \begin{subfigure}{0.48\textwidth}
    {\includegraphics[width=\linewidth]{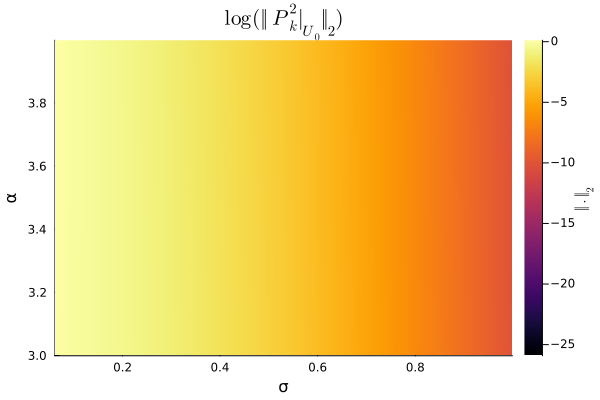}}
  \end{subfigure}
  \hfill
  \begin{subfigure}{0.48\textwidth}
    {\includegraphics[width=\linewidth]{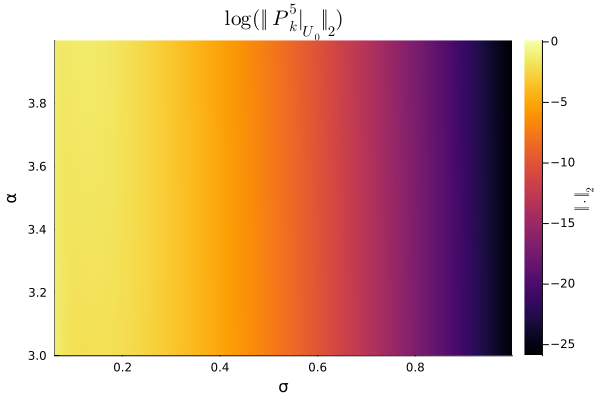}}
  \end{subfigure}
  \caption{Upper bounds on $\log(||P^N_k|_{\mathcal{U}_0}||_2)$ for $N=2$ and $N=5$ as a function of $\alpha$ and $\sigma$.}
  \label{fig:mixing}
\end{figure}



Beyond the specific examples considered in this work, our approach contributes to the broader understanding of random dynamical systems exhibiting nonuniformly hyperbolic behavior. By combining transfer operator techniques, Fourier-based discretization, and validated numerics, we are able to rigorously analyze statistical properties and dynamical transitions even in the absence of uniform expansion.

Future directions include the study of finer statistical features, such as variance, limit theorems, and response functions, following the lines of \cite{B2016, B2018}, as well as a detailed spectral analysis of the associated transfer operators. The latter can be approached using the pseudospectral and rigorous numerical techniques recently developed in \cite{blumenthal2025pseudospectralapproachrigorousnumerical}.

\section{Logistic Maps with noise}\label{sec:rand_skew}
In this section we define more precisely the systems we are going to consider.
In this work, we consider  the following family of maps,  generalization of the quadratic family:
\begin{align}
	T_{\alpha,\beta}(x)=\beta-(1+\beta)|x|^{\alpha},\label{p1}
\end{align}
where $\alpha \geq 1$, $\beta \in (-1,1]$ and $x\in[-1,1]$.\\
We introduce the notations we use for the Gaussian distribution, which will be necessary to define the random dynamical system.

\begin{definition}
	A random variable \(\Omega\) is said to have a \textbf{Gaussian distribution} if its density function is given by
	\[
		\rho_{\sigma}(x) = \frac{1}{\sigma \sqrt{2\pi}} e^{-x^{2}/(2\sigma^2)}.
	\]
\end{definition}

While the Gaussian distribution has unbounded support, the empirical rule (also known as the 68–95–99.7 rule) indicates that the standard deviation effectively quantifies the typical size of noise fluctuations.

\begin{definition}\label{c10}
	Define \(\tau: \mathbb{R} \rightarrow [-1,1)\) by
	\[
		\tau(x) = x - 2\left\lfloor \frac{x+1}{2} \right\rfloor.
	\]
	We refer to \(\tau\) as the \textbf{periodic boundary condition}.
\end{definition}

\begin{remark}
	The function \(\tau\) maps \(\mathbb{R}\) onto \([-1,1)\) by identifying points modulo the equivalence relation
	\[
		x \sim y \quad \text{if } x - y = 2k, \quad k \in \mathbb{Z}.
	\]
\end{remark}

With these definitions in place, we can now describe the dynamical system with Gaussian noise
\begin{align}
	X_{n+1}=\tau(T_{\alpha,\beta}(X_n)+\Omega_{\sigma}(n)),\label{p3}
\end{align}
where $\Omega_{\sigma}(n)$ are random variables i.i.d. with Gaussian distribution and  $\tau$ the boundary condition.

\section{Lyapunov exponent for the r.d.s.}\label{sec:randskew}
In this section, we recall basic definitions for random dynamical systems with additive noise on the interval \( [-1,1] \), including the Lyapunov exponent.

We model the system as a skew product:
\[
F: \mathbb{R}^{\mathbb{N}} \times [-1,1] \rightarrow \mathbb{R}^{\mathbb{N}} \times [-1,1], \quad
(\omega, x) \mapsto \left( s(\omega), \tau(T_{\alpha,\beta}(x) + \omega_0) \right),
\]
where \( T_{\alpha,\beta} : [-1,1] \to \mathbb{R} \) is a measurable map, \( \tau : \mathbb{R} \to [-1,1] \) denotes the periodic projection (e.g., modulo \(2\)), and \( s : \mathbb{R}^{\mathbb{N}} \to \mathbb{R}^{\mathbb{N}} \) is the left shift: \( s(\omega)_n = \omega_{n+1} \).

The space \( \Omega := \mathbb{R}^{\mathbb{N}} \), representing the realizations of the noise process, is endowed with the product topology and the corresponding Borel \(\sigma\)-algebra. We define the product probability measure \( \nu := \rho_\sigma^{\mathbb{N}} \), where \( \rho_\sigma \) is the Gaussian measure \( \mathcal{N}(0, \sigma^2) \) on \( \mathbb{R} \). We refer to \( (\Omega, \nu) \) as the \textbf{noise space}.

The random dynamical system is then described by the iterates of \( F \), acting on the product space \( \Omega \times [-1,1] \), with randomness encoded in the noise sequence \( \omega \in \Omega \).

We are interested in measuring the average expansion along orbits of the skew product, that is, in studying the limit
\[
\lim_{n \to \infty} \frac{1}{n} \log \left| \frac{\partial F^n}{\partial x}(\omega, x) \right|.
\]
We note that the function \(\tau\) is piecewise linear with jump discontinuities at the points \(x = 2k - 1\), \(k \in \mathbb{Z}\), where the floor function in its definition causes discontinuous jumps. Consequently, \(\tau\) is not differentiable at these points. However, since the set of discontinuities is countable, it has Lebesgue measure zero. Therefore, the chain rule
\[
\frac{d}{dx} \tau(T_{\alpha,\beta}(x) + y) = T'_{\alpha,\beta}(x)
\]
holds for almost every \(x\), provided that \(T_{\alpha,\beta}\) is differentiable and the argument \(T_{\alpha,\beta}(x) + y\) avoids the discontinuity set of \(\tau\) for almost all \(y\). Since the noise distribution is absolutely continuous, this is indeed the case for almost every realization.

Thus, letting \(\mathrm{proj}_x\) denote the projection onto the second coordinate, and setting \(X_i := \mathrm{proj}_x(F^i(\omega,x))\), we find
\[
\frac{1}{n} \log \left| \frac{\partial F^n}{\partial x}(\omega, x) \right| = \frac{1}{n} \sum_{i=0}^{n-1} \log |T'_{\alpha, \beta}(X_i)|.
\]

\begin{definition}
	A probability measure \(\mu\) on \([-1, 1]\) is called \textbf{stationary} if the product measure \(\nu \times \mu\) is invariant under the skew product \(F\).
\end{definition}

\begin{definition}
	A stationary measure \(\mu\) is said to be \textbf{ergodic} if \(\nu \times \mu\) is ergodic with respect to the skew product \(F\).
\end{definition}

\begin{lemma}\label{c41}
	Suppose the random dynamical system admits a unique stationary measure \(\mu_{\sigma}\). Let \(\phi: \mathbb{R}^{\mathbb{N}} \times [-1,1] \to \mathbb{R}\) be a measurable function such that \(\phi(\omega,x)\) depends only on \(x\), i.e., there exists \(\bar{\phi}\)  such that \(\phi(\omega,x) = \bar{\phi}(x)\). Suppose $\bar{\phi}\in L^1(\mu_\sigma)$, then for \(\mu_{\sigma}\)-almost every \(x_0\) and for \(\nu\)-almost every \(\omega\),
	\[
	\lim_{n \to \infty} \frac{1}{n} \sum_{i=0}^{n-1} \phi(F^i(\omega, x_0)) = \int \bar{\phi} \, d\mu_{\sigma}.
	\]
\end{lemma}

\begin{proof}
	The ergodic decomposition theorem for stationary measures \cite[Theorem 5.14]{V2014} asserts that every stationary measure is a convex combination of ergodic stationary measures. Since \(\mu_{\sigma}\) is unique, it must be ergodic. Because \(\bar{\phi} \in L^1(\mu_\sigma)\), it follows that \(\phi \in L^1(\nu \times \mu_\sigma)\). Therefore, by the Birkhoff Ergodic Theorem,
	\[
	\lim_{n \to \infty} \frac{1}{n} \sum_{i=0}^{n-1} \phi(F^i(\omega, x_0)) = \int \phi(\omega, x) \, d(\nu \times \mu_\sigma) = \int \bar{\phi}(x) \, d\mu_\sigma
	\]
	for \((\omega, x_0)\) in a set of full \(\nu \times \mu_\sigma\)-measure.
\end{proof}

In our setting, the function \(\phi(\omega, x) = \log |T'_{\alpha,\beta}(x)|\) depends only on the second coordinate and satisfies the integrability condition required by Lemma~\ref{c41}. Therefore, the Birkhoff average converges almost surely to the space average with respect to \(\mu_\sigma\), and we obtain the following characterization of the Lyapunov exponent:
\begin{align}
	\lim_{n \to \infty} \frac{1}{n} \log \left| \frac{\partial F^n}{\partial x}(\omega, x) \right| = \int \log |T'_{\alpha,\beta}| \, d\mu_\sigma. \label{c42}
\end{align}
Thus, when the stationary measure is unique, the Lyapunov exponent can be computed by integrating the logarithm of the derivative of the map with respect to the stationary measure.

As \(\sigma\) varies, we are interested in how the Lyapunov exponent behaves as a function of noise amplitude and its dependence on the parameters $\alpha$ and $\beta$, i.e., studying the following function
\[
\lambda(\alpha,\beta,\sigma) = \int_{-1}^{1} \log |T'_{\alpha,\beta}| \, d\mu_\sigma;
\]
the next sections will be devoted to proving it is well-defined and some of its properties.

\begin{definition}\label{f40}
	We say that a system exhibits a \textbf{transition of the Lyapunov exponent} if there exist \(0 < \sigma_1 < \sigma_2\) such that for all \(\sigma \geq \sigma_1\), the system admits a unique stationary measure with density \(f_{\sigma}\), and the Lyapunov exponent transitions continuously changing sign, that is, \(\lambda(\sigma_1) \cdot \lambda(\sigma_2) < 0\). We say the system exhibits a \textbf{multiple transition} if several such sign transitions occur.
\end{definition}

\section{The annealed transfer operator}\label{c21}
Following a well-established strategy, we want to reframe the problem of finding a stationary measure 
as a fixed point problem for an appropriate operator acting on measures, the annealed transfer operator.
In this section we define the operator and show that is regularizing, i.e., preserves a space of 
regular densities. We start by investigating the properties of convolution with a Gaussian kernel.

\begin{definition}\label{c51}
	Denote by
	\begin{align*}
		\widehat{f}(x)=\begin{cases}
			f(x), & x\in[-1,1], \\[0.3em]
			0, & \text{other case}. \\
		\end{cases}
	\end{align*}
	the function that extend $f$ by $0$ outside its intervals of definition.
	Given a probability measure $\mu$ on $[-1,1]$, we define its extension $\widehat{\mu}$ on $\mathbb{R}$ as the unique measure $\widehat{\mu}$ on $\mathbb{R}$ such that
	\begin{align*}
		\widehat{\mu}(A)=\mu(A\cap [-1,1])
	\end{align*}
	for all $A$ measurable in $\mathbb{R}$.
\end{definition}

We will define an operator as the convolution of a measure with $\rho_{\sigma}$, and we will examine some of its regularizing properties. The most important of these is Proposition $\ref{c16}$, which tells us that the convolution of a measure with $\rho_{\sigma}$ is absolutely continuous with respect to the Lebesgue measure.

\begin{definition}\label{def:convolution}
Let $\mu$ be any probability measure on $[-1,1]$, and $\rho_{\sigma}$ the Gaussian distribution; their convolution $\rho_{\sigma}*\widehat{\mu}$ is the unique probability measure on $\mathbb{R}$ such that
\[
\rho_\sigma * \widehat{\mu}(A)
=
\int_A \int_{-1}^{1} \rho_\sigma(y-z)\, d\widehat{\mu}(z)\, dm,
\]
\end{definition}
where $m$ is nonnormalized Lebesgue measure.
\begin{remark}
Definition~\ref{def:convolution} is equivalent to
\[
(\rho_{\sigma}*\widehat{\mu})(A)
=\int_{\mathbb R}\rho_{\sigma}(y)\,\widehat{\mu}(A-y)\,dm(y),
\qquad A-y:=\{x-y:\ x\in A\},
\]
This form highlights that the noise acts on the system by averaging over translations.
\end{remark}

The following lemma, provided without proof, states some properties of extension and convolution that follow 
directly from the definition.
\begin{lemma}\label{c6}
Let $\mu$ be a probability measure on $[-1,1]$ 
\begin{enumerate}
	\item if $\mu= fdm$ we have that $\widehat{\mu}=\widehat{f}dm$,
	\item $\rho_{\sigma}*\widehat{\mu}(\mathbb{R})=1$,
	\item if $\mu=fdm$ then $\rho_{\sigma}*\widehat{\mu}$ has density $\rho_{\sigma}*\widehat{f}$ (i.e. $\rho_{\sigma}*\widehat{\mu}=(\rho_{\sigma}*\widehat{f})dm$).
\end{enumerate}
\end{lemma}

The proofs of these lemmas follow immediately from the definition. We show now a simple result, proving that the convolution with a Gaussian kernel maps measures into the more 
regular space of absolutely continuous invariant measures.
\begin{proposition} \label{c16}
Let $\mu$ be a probability measure on $[-1,1]$, then $\rho_{\sigma}*\widehat{\mu}$ is a probability measure on $\mathbb{R}$, absolutely continuous with respect to Lebesgue.
\end{proposition}
\begin{proof}
Let 
\[
g(x) = \int_{-1}^1 \rho_\sigma(x-z)\, d\widehat{\mu}(z)
\]
Then 
\[
\rho_{\sigma}*\widehat{\mu} = g(x) dm.
\]

\end{proof}

\begin{definition}\label{f43}
Let $\mathcal{B}([-1,1])$ be the Borel $\sigma$-algebra on $[-1,1]$, defined as the restriction of the Borel $\sigma$-algebra of $\mathbb{R}$ to $[-1,1]$.\\ Let $\mu$ be a measure on $([-1,1],\mathcal{B}([-1,1]))$, a measurable map $T:[-1,1]\rightarrow [-1,1]$ is called \textbf{non-singular} with respect to $\mu$ if,	 $\mu(T^{-1}(A))=0$ for all $A\in \mathcal{B}([-1,1])$ such that $\mu(A)=0$.
\end{definition}

We define the Annealed Transfer Operator associated to the system with noise (see e.g. \cite{L2004} or \cite{G1} for an introduction to the topic).
\begin{definition}\label{c3}
	Let $T:[-1,1]\rightarrow [-1,1]$ be a measurable map. The map $T$ induces an operator on $L:\mathcal{SM}([-1,1])\rightarrow \mathcal{SM}([-1,1])$ where $\mathcal{SM}([-1,1])$ is the space of signed measures on $[-1,1]$, defined in the following way: if $\mu\in \mathcal{SM}([-1,1])$ then
	\begin{align*}
		L\mu (A)=\mu (T^{-1} A)
	\end{align*}
	for all measurable set $A$. This operator is called the \textbf{transfer operator} associated to $T$.\\
	The space of Lebesgue absolutely continuous measures is a vector subspace of \\
	$\mathcal{SM}([-1,1])$; if $T$ is non-singular with respect to Lebesgue then $L$ preserves this subspace of absolutely continuous measures and induces an operator from $L^1([-1,1])$ into itself called the \textbf{Perron-Frobenius} operator. We will denote by $P$ the Perron-Frobenius operator.
\end{definition}
Remark that, even though the dynamical system \(T\) is defined on the interval \([-1,1]\), the addition of random noise may result in iterates lying outside this interval. The next definition fixes the notation for the action of
the boundary condition on measures.
\begin{definition}
	Let $\tau:\mathbb{R}\rightarrow [-1,1]$ be the boundary condition as in the Definition \ref{c10}.
	We denote by $\tau_*$ the push-forward map acting on measures by
	\begin{align*}
		(\tau_*\mu)(A)=\mu(\tau^{-1}(A)),
	\end{align*}
i.e., the transfer operator associated to the boundary condition. By abuse de notation $\tau_{*}$ will denote also the maps that induces on densities i.e., if $\mu$ has density $g$, then $\tau_{*}(g)$ is the density of $\tau_{*}\mu$. 
\end{definition}
\begin{definition}
	The annealed transfer operator $L_{\sigma}$ associated to the system with noise is defined by
	\begin{align*}
		L_{\sigma}\mu=\tau_*(\rho_{\sigma}*\widehat{L\mu}).
	\end{align*}
\end{definition}

The next corollary follows directly from Proposition \ref{c16} and the definitions.
\begin{corollary} \label{c17}
	The operator $L_{\sigma}$ induces an operator $P_{\sigma}:L^1(m)\to L^1(m)$, acting on densities such that
	\begin{align*}
		P_{\sigma}f=\tau_{*}(\rho_{\sigma}*\widehat{Pf}).
	\end{align*}
	We refer to this operator as the \textbf{annealed Perron-Frobenius operator}.
\end{corollary}
\begin{proof}
	Let $\mu=fdm$ be a absolutely continuous probability measure with density $f$.\\
	By Definition \ref{c3} we have that
	\begin{align*}
		L\mu=Pf dm. 
	\end{align*}
	By Lemma $\ref{c6}$, we have that
	\begin{align*}
		\rho_{\sigma}*\widehat{L\mu}=(\rho_{\sigma}*\widehat{Pf})d m,
	\end{align*}
	where the Lebesgue measure on the right-hand side is defined on $\mathbb{R}$. Therefore, we have that
	\begin{align*}
		\tau_*(\rho_{\sigma}*\widehat{L\mu})=\tau_*(\rho_{\sigma}*\widehat{Pf})d m,
	\end{align*}
	where on the right-hand side $m$ is defined on $[-1,1]$. Let $P_{\sigma}f=\tau_*(\rho_{\sigma}*\widehat{Pf})$, then $L_{\sigma}\mu=P_{\sigma}fd m$.
\end{proof}

The following result connects the annealed transfer operator with the skew-product representation 
of a random dynamical system \cite[Section 5.3]{V2014}.
\begin{lemma}\label{def:stationary}
Let $T$ be any non-singular dynamical map of the interval and let the random system with Gaussian noise be defined as:
\begin{align}
	X_{n+1}=\tau(T(X_n)+\Omega_{\sigma}(n)),
\end{align}
where $\Omega_{\sigma}(n)$ are random variables i.i.d. with Gaussian distribution with standard deviation $\sigma$ and  $\tau$ are periodic boundary conditions. 
If $\mu_{\sigma}$ is a fixed point for $L_{\sigma}$, then $\mu_{\sigma}$ is a stationary measure 
for the random dynamical system.
\end{lemma}

\begin{remark}
	If $P_{\sigma}$ is  the annealed Perron-Frobenius operator operating on densities, and $f_{\sigma}$ is a fixed point of this operator
	\begin{align*}
		P_{\sigma}f_{\sigma}=f_{\sigma}
	\end{align*}
	then $\mu_{\sigma}=f_{\sigma}d m $ (where $m$ is the Lebesgue measure) is a stationary measure.
\end{remark}

The discussion in subsection \ref{sec:rand_skew} can now be restated in the following theorem; we refer also to \cite[Proposition 5.4]{V2014}.
\begin{theorem}[Birkhoff Ergodic Theorem for random dynamical system]\label{c18}
	Suppose that $L_{\sigma}$ has a unique stationary measure $\mu_{\sigma}$ and let $\phi \in L^{1}(\mu_{\sigma})$. Then, for $\mu_{\sigma}$ almost every initial condition $X_0 = x_0$ and with probability $1$
	\begin{align*}
		\displaystyle \lim_{n\rightarrow +\infty}\frac{1}{n}\sum_{i=0}^{n-1}\phi(X_i)=\int \phi d\mu_{\sigma}.
	\end{align*}
\end{theorem}

\section{Regularization properties of the convolution}
In section \ref{c21}, we defined the operator $P_{\sigma}$. Our goal is to approximate this operator by a finite rank operator and use a matrix associated with this finite rank approximation to enclose a fixed point (section \ref{c23}). To achieve this, we will introduce and use the Fourier transform, its truncation and examine some important properties of the finite rank matrix obtained by projecting through truncation (Proposition \ref{f5}, Corollary \ref{thm:decay_fourier}), often called the Galerkin approximation.

Let $A$, $B$ be normed vector spaces and let $Q:A\rightarrow B$ be a linear operator.\\ We use the usual operator norm, defined by
\begin{align*}
	\|Q\|_{A\rightarrow B}=\sup_{f\in A, \|f\|_{A}\leq 1}\|Qf\|_{B}.
\end{align*}
\begin{definition}
	Define $N_{\sigma}:L^1\rightarrow L^{1}$ such that
	\begin{align*}
		N_{\sigma}f(x)=(\tau_*(\rho_{\sigma}*\widehat{f}))(x).
	\end{align*}
\end{definition}
We will see that the operator $N_{\sigma}$ relates the Perron-Frobenius operator $P$ of the dynamical system $T$ with the operator $P_{\sigma}$, defined in Lemma $\ref{c17}$, acting on densities.\\
Let $f\in L^1([-1,1])$. We have that
\begin{align}
	N_{\sigma}f(x) & =(\rho_{\sigma}*\widehat{f})(\tau^{-1}(x)) \nonumber \\ & =\displaystyle \sum _{j=-\infty}^{+\infty}\int_{-1}^{1}\rho_{\sigma}(x+2j-y)f(y)dy      \label{c49}.
\end{align}
Note that, if $f\in L^1([-1,1])$, then
\begin{align*}
	N_{\sigma}Pf(x)=\tau_*((\rho_{\sigma}*\widehat{Pf})(x))=P_{\sigma}f(x)
\end{align*}
and therefore $N_{\sigma}P=P_{\sigma}$.\\
We summarize some properties of the operator $N_{\sigma}$ in the next lemma.
\begin{lemma}\label{lem:regN}
Let $\rho_{\sigma}$ be the Gaussian density with mean $0$ and standard deviation $\sigma$ and $\tau$ be the periodic boundary condition and let
\[
N_{\sigma}f=(\tau_*(\rho_{\sigma}*\widehat{f})),
\]
for each function $f\in L^1(m)$.
Then 
\begin{enumerate}
	\item $N_{\sigma}f(x) = \int_{-1}^{1}(\tau_*\rho_{\sigma})(x-y)f(y)dy$, \label{item1}
	\item $N_{\sigma}$ is a bounded operator from $L^1\to L^{\infty}$,
	\item $||N_{\sigma}||_{L^1\rightarrow L^1}\leq 1$,
	\item $(N_{\sigma}f)'(x) = \int_{-1}^1 \tau_*(\rho'_{\sigma})(x-y)f(y)dy$,
	\item $N_{\sigma}$ is a bounded operator from $L^1\to BV([-1, 1])$,
	\item If \( f \geq 0 \), \( \int_{-1}^1 f = 1 \), then \( N_\sigma f(x) > 0 \) for all \( x \in [-1,1] \), and
\[
N_\sigma f(x) \geq \frac{1}{\sigma \sqrt{2\pi}} e^{-1/(2\sigma^2)} \quad \text{for all } x \in [-1,1].
\]
\end{enumerate}
\end{lemma}
\begin{proof} 
Recall that the periodization of a function $g\colon\mathbb R\to\mathbb R$ is
\[
(\tau_* g)(x):=\sum_{j\in\mathbb Z} g(x+2j),\qquad x\in[-1,1].
\]
Since $\rho_\sigma$ is rapidly decaying, the series defining $\tau_*\rho_\sigma$ converges absolutely and uniformly on $[-1,1]$.
Moreover, the same holds for the derivative series $\displaystyle \sum_{j\in\mathbb Z}\rho_\sigma'(x+2j)$, hence $\tau_*\rho_\sigma\in C^1([-1,1])$ and
\[
(\tau_*\rho_\sigma)'(x)=\sum_{j\in\mathbb Z}\rho_\sigma'(x+2j)=\tau_*(\rho_\sigma')(x).
\]

\smallskip\noindent
By the definition of convolution with the measure $\widehat f=f(y)\,dy$ supported on $[-1,1]$, we have for all $x\in\mathbb R$,
\[
(\rho_\sigma*\widehat f)(x)=\int_{-1}^1 \rho_\sigma(x-y)\,f(y)\,dy.
\]
Therefore, for $x\in[-1,1]$,
\begin{align*}
N_\sigma f(x)
=(\tau_*(\rho_\sigma*\widehat f))(x)
&=\sum_{j\in\mathbb Z} (\rho_\sigma*\widehat f)(x+2j)\\
&=\sum_{j\in\mathbb Z}\int_{-1}^1 \rho_\sigma(x+2j-y)\,f(y)\,dy\\
&=\int_{-1}^1 \Bigl(\sum_{j\in\mathbb Z}\rho_\sigma(x-y+2j)\Bigr)\,f(y)\,dy\\
&=\int_{-1}^1 (\tau_*\rho_\sigma)(x-y)\,f(y)\,dy.
\end{align*}
The interchange of sum and integral is justified by dominated convergence, since
$\displaystyle \sum_{j\in\mathbb Z}\rho_\sigma(x-y+2j)=(\tau_*\rho_\sigma)(x-y)\le \|\tau_*\rho_\sigma\|_\infty$
and $f\in L^1$.

From item (\ref{item1}),
\[
|N_\sigma f(x)|\le \int_{-1}^1 (\tau_*\rho_\sigma)(x-y)\,|f(y)|\,dy
\le \|\tau_*\rho_\sigma\|_\infty\,\|f\|_{L^1},
\]
hence $\|N_\sigma f\|_\infty\le \|\tau_*\rho_\sigma\|_\infty\|f\|_{L^1}$.

First note that $N_\sigma$ is positive. Moreover, for any $y\in[-1,1]$,
\[
\int_{-1}^1 (\tau_*\rho_\sigma)(x-y)\,dx
=\sum_{j\in\mathbb Z}\int_{-1}^1 \rho_\sigma(x-y+2j)\,dx
=\int_{\mathbb R}\rho_\sigma(u)\,du
=1.
\]
Thus for $g\ge 0$,
\[
\int_{-1}^1 N_\sigma g(x)\,dx
=\int_{-1}^1\int_{-1}^1 (\tau_*\rho_\sigma)(x-y)\,g(y)\,dy\,dx
=\int_{-1}^1 g(y)\,dy.
\]
For general $f$, positivity gives $|N_\sigma f|\le N_\sigma|f|$, hence
$\|N_\sigma f\|_{L^1}\le \|N_\sigma|f|\|_{L^1}=\|f\|_{L^1}$, i.e.\ $\|N_\sigma\|_{L^1\to L^1}\le 1$.

Since $\tau_*\rho_\sigma\in C^1$ and $(\tau_*\rho_\sigma)'=\tau_*(\rho_\sigma')$ is bounded, we may differentiate under the integral sign:
\[
(N_\sigma f)'(x)=\int_{-1}^1 (\tau_*\rho_\sigma)'(x-y)\,f(y)\,dy
=\int_{-1}^1 \tau_*(\rho_\sigma')(x-y)\,f(y)\,dy.
\]

As $N_\sigma f\in C^1([-1,1])$, it belongs to $BV([-1,1])$ and
\begin{align*}
\textrm{Var}(N_\sigma f)&=\int_{-1}^1 |(N_\sigma f)'(x)|\,dx\\
&\leq \int_{-1}^1\int_{-1}^1 |\tau_*(\rho_\sigma')|(x-y)\,|f(y)|\,dy\,dx
\leq \|\tau_*(\rho_\sigma')\|_{L^1}\,\|f\|_{L^1}.
\end{align*}
Consequently,
\[
\|N_\sigma f\|_{BV}\;\le\;\|N_\sigma f\|_{L^1}+\textrm{Var}(N_\sigma f)
\;\le\;\bigl(1+\|\tau_*(\rho_\sigma')\|_{L^1}\bigr)\,\|f\|_{L^1}.
\]
(One may also use the explicit bound $\|\tau_*(\rho_\sigma')\|_{L^1}\le \int_{\mathbb R}|\rho_\sigma'(u)|\,du
=\sqrt{\frac{2}{\pi}}\frac{1}{\sigma}$.)

Since $\tau_*\rho_\sigma$ is $2$--periodic and $\tau_*\rho_\sigma(t)\ge \rho_\sigma(t)$ for all $t\in\mathbb R$,
we have
\[
\inf_{t\in\mathbb R}\tau_*\rho_\sigma(t)
=\inf_{t\in[-1,1]}\tau_*\rho_\sigma(t)
\ge \inf_{t\in[-1,1]}\rho_\sigma(t)=\rho_\sigma(1)
=\frac{1}{\sigma\sqrt{2\pi}}e^{-1/(2\sigma ^2) }.
\]
If $f\ge 0$ and $\int_{-1}^1 f=1$, then by (\ref{item1}),
\[
N_\sigma f(x)=\int_{-1}^1 (\tau_*\rho_\sigma)(x-y)\,f(y)\,dy
\ge \Bigl(\inf_{t\in\mathbb R}\tau_*\rho_\sigma(t)\Bigr)\int_{-1}^1 f(y)\,dy
\ge \rho_\sigma(1),
\]
so in particular $N_\sigma f(x)>0$ for all $x\in[-1,1]$ and the claimed uniform lower bound holds.
\end{proof}

The fact that $N_{\sigma}$ is regularizing $L^1$ into $BV$ has important consequences

\begin{lemma}\label{lemma:existence}
Let $P_{\sigma}:L^1([-1,1])\rightarrow L^1([-1,1])$, $P_{\sigma} = N_{\sigma} P$ be the annealed Perron-Frobenius operator associated to the random dynamical system with Gaussian noise of variance $\sigma$.
\begin{itemize}
	\item $||P_{\sigma}||_{L^1\rightarrow L^1}\leq 1$,
	\item $P_{\sigma}$ maps $L^1([-1,1])$ into $BV([-1,1])$,
	\item $P_{\sigma}$ is a compact operator from $L^1([-1,1])$ to $L^1([-1,1])$,
	\item $P_{\sigma}$ has at least a fixed point, i.e., there exists at least a stationary density. 
\end{itemize}
\end{lemma}
\begin{proof}
By Lemma \ref{lem:regN} we have that 
\[
||P_{\sigma} f||_{L^1} = ||N_{\sigma} Pf||_{L^1}\leq ||N_{\sigma}||_{L^1\rightarrow L^1}||Pf||_{L^1}\leq ||f||_{L^1}
\]
since $||P||_{L^1\rightarrow L^1}\leq 1$, as it is the Perron-Frobenius operator of a non-singular transformation.

Moreover,
\[
||(P_{\sigma}f)'||_{L^1}\leq ||\tau_* \rho'_{\sigma}||_{L^{\infty}}||Pf||_{L^1}\leq ||\tau_* \rho'_{\sigma}||_{L^{\infty}}||f||_{L^1}  
\]
therefore $P_{\sigma}$ maps a bounded sequence in $L^1$ into a bounded sequence in $BV$, which proves that \( P_\sigma \) maps bounded sets in \( L^1 \) into bounded subsets of \( BV \). 
Since the inclusion \( BV([-1,1]) \hookrightarrow L^1([-1,1]) \) is compact, this implies that \( P_\sigma : L^1 \to L^1 \) is compact.

From this same compactness inclusion property, we have then, following the classical Krylov-Bogolyubov argument (see e.g.  \cite{KrylovBogolyubov1937} or \cite{1989DsI}) that the sequence 
\[
f_N = \frac{1}{N}\sum_{j=0}^{N-1} P_{\sigma}^j \chi_{[-1,1]}
\]
has a subsequence that converges in \( L^1 \) to a fixed point of \( P_\sigma \), i.e., a stationary density.
\end{proof}
\begin{remark}
The Krylov-Bogolyubov argument ensures the existence of a fixed point, but does not guarantee uniqueness.
\end{remark}

\begin{definition}\label{def:average0} Let
\begin{align*}
	\mathcal{U}_0=\{f\in L^2([-1,1])|\int fdm=0\}.
\end{align*}
We call $\mathcal{U}_0$ the vector subspace of average $0$ functions. We say $P_{\sigma}$ \textbf{contracts the space of average $0$ functions in} $L^2$ if
\begin{align*}
	\|P_{\sigma}^n|_{\mathcal{U}_0}\|_{L^2\rightarrow L^2}\leq C \theta^n
\end{align*}
for constants $C>0$, $0<\theta<1$.
\end{definition}

We prove now a version of Doeblin's argument adapted to our discussion with densities and operators.
\begin{proposition} \label{prop:Doeblin}
	Suppose that there exists $c>0$ such that for every density $f$ we have $P_{\sigma}f>c$. Then
	\begin{align*}
		\|P_{\sigma}^n |_{\mathcal{U}_0}\|_{L^1\rightarrow L^1}\leq (1-2c)^{n}.
	\end{align*}
\end{proposition}
\begin{proof}
From Lemma~\ref{c17}, we know that \( P_\sigma \) acts on densities. Let \( h \in \mathcal{U}_0 \) with \( \|h\|_{L^1} = 1 \). Then \( h \) can be written as \( h = f - g \), where \( f, g \geq 0 \), \( \int f \, dm = \int g \, dm \), and \( \|f\|_{L^1} + \|g\|_{L^1} = 1 \). Hence, \( \|f\|_{L^1} = \|g\|_{L^1} = \frac{1}{2} \).

By hypothesis, since \( f \) and \( g \) are densities with mass \( \frac{1}{2} \), we have
\[
P_\sigma f \geq \frac{c}{2}, \quad P_\sigma g \geq \frac{c}{2} \quad \text{almost everywhere}.
\]

Since \( P_\sigma \) is a linear operator, we have
\[
P_\sigma h = P_\sigma f - P_\sigma g,
\]
and thus
\begin{align*}
\|P_\sigma h\|_{L^1} 
&\leq \left\|P_\sigma f - \frac{c}{2} \right\|_{L^1} + \left\|P_\sigma g - \frac{c}{2} \right\|_{L^1} \\
&= \int_{[-1,1]} \left|P_\sigma f - \frac{c}{2} \right| dm + \int_{[-1,1]} \left|P_\sigma g - \frac{c}{2} \right| dm.
\end{align*}

Since \( P_\sigma f \geq \frac{c}{2} \) and \( P_\sigma g \geq \frac{c}{2} \), we may drop the absolute values:
\begin{align*}
\|P_\sigma h\|_{L^1}
&= \int_{[-1,1]} \left(P_\sigma f - \frac{c}{2} \right) dm + \int_{[-1,1]} \left(P_\sigma g - \frac{c}{2} \right) dm \\
&= \|P_\sigma f\|_{L^1} - c + \|P_\sigma g\|_{L^1} - c \\
&\leq \frac{1}{2} - c + \frac{1}{2} - c = 1 - 2c.
\end{align*}

This shows that
\[
\|P_\sigma|_{\mathcal{U}_0}\|_{L^1 \to L^1} \leq 1 - 2c,
\]
and by iteration,
\[
\|P_\sigma^n|_{\mathcal{U}_0}\|_{L^1 \to L^1} \leq (1 - 2c)^n,
\]
as claimed.
\end{proof}

\begin{proposition} \label{prop:mixinggaussian}
Let \( \rho_{\sigma} \) be the Gaussian noise kernel. Then
\[
\|P_{\sigma}^n |_{\mathcal{U}_0}\|_{L^2 \to L^2} \leq C \theta^n,
\]
where \( C = \left(\sqrt{2} + 3\sqrt{2}\rho_{\sigma}(0)\right)^{1/2} \) and \( \theta = (1 - 2c)^{1/2} \in (0,1) \).
\end{proposition}

\begin{proof}
From the positivity of the Gaussian kernel and the definition of \( P_\sigma \), from Lemma \ref{lem:regN} we have:
\[
P_\sigma f(x) = \int_{-1}^1 \sum_{k\in \mathbb{Z}} \rho_\sigma(x + 2k - y) P f(y)\,dy \geq c := \frac{1}{\sigma \sqrt{2\pi}} e^{-1/(2\sigma^2)},
\]
for all densities \( f \). Thus, Proposition~\ref{prop:Doeblin} applies and yields:
\[
\|P_\sigma^n h\|_{L^1} \leq (1 - 2c)^n \|h\|_{L^1}, \quad \forall h \in \mathcal{U}_0.
\]

Now estimate the \( L^2 \) norm using interpolation. Due to typographical limitations, we use the notation $(\cdot)^\wedge$ instead of $\widehat{\cdot}$ for long expressions such as $P_\sigma(P_\sigma^{n-1}h)$:
\begin{align}
\|P_\sigma^n h\|_{L^2}^2 
&\leq \|P_\sigma^n h\|_{L^1} \cdot \|P_\sigma^n h\|_{L^\infty} \nonumber \\
&= \|P_\sigma^n h\|_{L^1} \cdot \|P_\sigma(P_\sigma^{n-1} h)\|_{L^\infty} \nonumber \\
&\leq \|P_\sigma^n h\|_{L^1} \cdot \| \tau_* \rho_\sigma * (P(P_\sigma^{n-1} h))^\wedge \|_{L^\infty} \nonumber \\
&\leq \|P_\sigma^n h\|_{L^1} \cdot \|\tau_* \rho_\sigma\|_{L^\infty} \cdot \|P(P_\sigma^{n-1} h)\|_{L^1} \nonumber \\
&\leq (1 - 2c)^n \cdot (1 + 3\rho_\sigma(0)) \cdot \|h\|_{L^1}, \label{eq:L2bound}
\end{align}
where the last inequality uses Lemma \ref{lem:regN} and that \( \|P\|_{L^1 \to L^1} \leq 1 \).

Now normalize and pass to the operator norm:
\[
\|P_\sigma^n|_{\mathcal{U}_0}\|_{L^2 \to L^2} = \sup_{\substack{h \in \mathcal{U}_0\\ \|h\|_{L^2} \leq 1}} \|P_\sigma^n h\|_{L^2}
\leq \sup_{\|h\|_{L^2} \leq 1} \left((1 - 2c)^n (1 + 3\rho_\sigma(0)) \|h\|_{L^1} \right)^{1/2}.
\]
Using \( \|h\|_{L^1} \leq \sqrt{2} \|h\|_{L^2} \) for \( h \in \mathcal{U}_0 \subset L^2 \), we obtain:
\[
\|P_\sigma^n|_{\mathcal{U}_0}\|_{L^2 \to L^2} \leq (1 - 2c)^{n/2} \left(\sqrt{2} + 3\sqrt{2} \rho_\sigma(0) \right)^{1/2},
\]
as claimed. Since \(c>0 \), it follows that \( 0 < \theta = \sqrt{1 - 2c} < 1 \).
\end{proof}

In Proposition \ref{prop:mixinggaussian}, we demonstrated that $P_{\sigma}$ contracts the space of average $0$ functions in  $L^2$. We can then prove that $P_{\sigma}$ has a unique fixed point, as shown in the following proposition.
\begin{proposition}\label{prop:unique}
	If $P_{\sigma}$ contracts the space of average $0$ functions in  $L^2$, then $P_{\sigma}$ has a unique fixed point.
\end{proposition}
\begin{proof}
	We prove by contradiction. Let $\mu$ and $\nu$ be stationary measures. Since $L_{\sigma}\mu=\mu$ and $L_{\sigma}\nu=\nu$ we have that $\mu$ and $\nu$ are absolutely continuous with respect to Lebesgue measure (Lemma \ref{c16}), with densities $f$ and $g$ respectively. Now, $P_{\sigma}f=f$ and $P_{\sigma}g=g$, and since $P_{\sigma}$ contracts the space of average $0$ functions, we have that for any $n$
	\begin{align*}
		\|f-g\|_{L^2} & =\|P^{n}_{\sigma}f-P^{n}_{\sigma}g\|_{L^2} \\
		              & =\|P^{n}_{\sigma}(f-g)\|_{L^2}             \\
		              & \leq C\theta^n\|f-g\|_{L^2}.
	\end{align*}
	Take $N$ such that $C\theta^N<1$. The inequality above implies that $\|f-g\|_{L^1}=0$, which in turn implies that $\mu=\nu$. Therefore, $L_{\sigma}$ has a unique stationary measure, it follows that $P_{\sigma}$ has a unique fixed point.
\end{proof}

\section{Fourier theory and Galerkin approximation}
In this section, we formalize the use of Fourier analysis to study the action of the smoothing operator \( N_\sigma \). We define appropriate Banach spaces of sequences (e.g., \( \ell^\infty \), \( \ell^2 \), and exponentially decaying sequences \( \ell^{\exp}_\sigma \)), and show how the Fourier transform maps functions on \( [-1,1] \) into these spaces.

We prove that \( N_\sigma \) is diagonalized by the Fourier basis, with eigenvalues given by a diagonal operator \( D_\sigma \) (Proposition~\ref{f5}), and that the entries of \( D_\sigma \) decay exponentially fast (Theorem~\ref{thm:decay_fourier}). This result justifies the use of Galerkin-type finite-dimensional approximations, since high-frequency modes are exponentially suppressed.

\begin{definition}
We define the \textbf{Fourier transform} on \( [-1,1] \) as the map
\[
\mathcal{F}(f)[k] := \frac{1}{2} \int_{-1}^{1} f(x) e^{-k\pi i x} \, dx,
\]
for all \( k \in \mathbb{Z} \). The number \( \mathcal{F}(f)[k] \) is called the \( k \)-th Fourier coefficient of \( f \).
\end{definition}

\begin{definition}
Let \( \ell^\infty \) denote the space of bounded bilateral sequences:
\[
\ell^\infty = \{ (a_k)_{k \in \mathbb{Z}} : |a_k| \leq C \text{ for some } C > 0 \}.
\]
\end{definition}

\begin{definition}
For \( \sigma > 0 \), define the space of \textbf{exponentially decaying sequences}:
\[
\ell^{\exp}_\sigma = \left\{ (a_k)_{k \in \mathbb{Z}} : |a_k| \leq C e^{-\frac{\sigma^2 |k|}{2}} \text{ for some } C > 0 \right\}.
\]
\end{definition}

\begin{definition}
Define \( \ell^2 \), the space of square-summable sequences:
\[
\ell^2 = \left\{ (a_k)_{k \in \mathbb{Z}} : \sum_{k \in \mathbb{Z}} |a_k|^2 < \infty \right\},
\]
endowed with the norm \( \| (a_k) \|_{\ell^2} := \left( \sum |a_k|^2 \right)^{1/2} \). It forms a Hilbert space with respect 
to the inner product $<(a_k), (b_k)> = \sum a_k \bar{b}_k$. 
\end{definition}

\begin{lemma} \label{lem:fourier}
The Fourier transform \( \mathcal{F} \) on \( [-1,1] \) satisfies the following properties:
\begin{enumerate}
    \item \( \mathcal{F} : L^1([-1,1]) \to \ell^\infty \) is continuous with operator norm \( \| \mathcal{F} \| = 1 \). \label{it:continuity}
    \item \( \mathcal{F} \) is invertible on \( \ell^{\exp}_\sigma \), \label{it:analytic}
    \item \( \mathcal{F} : L^2([-1,1]) \to \ell^2 \) is an isometry, and the family \( \{ e^{-k\pi i x} \}_{k \in \mathbb{Z}} \) forms an orthonormal basis of \( L^2([-1,1]) \). \label{it:isometry}
\end{enumerate}
\end{lemma}
\begin{proof}
Proofs of items\ref{it:continuity} and \ref{it:isometry} can be found  in Theorem 9.6 and Theorem 9.13 of \cite{Rudin1987}, respectively. We only give a short proof of item \ref{it:analytic}. Since $\{a_k\}_{k\in \mathbb{Z}}$ is in $\ell^{\exp}_{\sigma}$, we have
	\begin{align*}
		\left|\displaystyle \sum _{k=-\infty}^{+\infty}a_k e^{\pi i kx} \right| & \leq \displaystyle \sum _{k=-\infty}^{+\infty}\left|a_k\right|<\infty,
	\end{align*}
	and from Weierstrass's M-test Theorem
	\begin{align*}
		\mathcal{F}^{-1}(\{a_k\})=\displaystyle \sum _{k=-\infty}^{+\infty}a_k e^{\pi i kx}
	\end{align*}
	converges uniformly to an analytic function.
\end{proof}

The operator $N_{\sigma}$ has an explicit representation that commutes with the Fourier transform.
\begin{proposition}\label{f5}
There exists an operator \( D_{\sigma} : \ell_\infty \to \ell_\infty \) such that
\[
D_{\sigma} \mathcal{F} = \mathcal{F} N_{\sigma}.
\]
\end{proposition}

\begin{proof}
We compute the \( k \)-th Fourier coefficient of \( N_\sigma f \):
\begin{align*}
\mathcal{F}(N_\sigma f)[k] 
&= \frac{1}{2} \int_{-1}^{1} \left( \sum_{j \in \mathbb{Z}} \int_{-1}^1 \rho_\sigma(x + 2j - y) f(y) \, dy \right) e^{-k\pi i x} \, dx \\
&= \frac{1}{2}\sum_{j \in \mathbb{Z}} \int_{-1}^1 \int_{-1}^1 \rho_\sigma(x + 2j - y) f(y) e^{-k\pi i x} \, dy dx \\
&= \frac{1}{2}\sum_{j \in \mathbb{Z}} \int_{-1}^1 \int_{-1}^1 \rho_\sigma(z + 2j) f(y) e^{-k\pi i (z + y)} \, dz dy \quad (z = x - y) \\
&= \frac{1}{2}\left( \sum_{j \in \mathbb{Z}} \int_{-1}^1 \rho_\sigma(z + 2j) e^{-k\pi i z} \, dz \right) \cdot \mathcal{F}(f)[k].
\end{align*}

Define the diagonal operator \( D_\sigma \) on \( \ell_\infty \) by
\[
D_\sigma[k,k] :=  \frac{1}{2}\int_{-\infty}^{+\infty} \rho_\sigma(z) e^{-k\pi i z} \, dz,
\]
so that \( \mathcal{F}(N_\sigma f)[k] = D_\sigma[k,k] \cdot \mathcal{F}(f)[k] \). Hence,
\[
\mathcal{F}(N_\sigma f) = D_\sigma \mathcal{F}(f),
\]
as claimed.
\end{proof}

We observe now that the coefficients $D[k, k]$ can be obtained by a classical computation, i.e. that 
\[
\int_{-\infty}^{+\infty}\rho_{\sigma}(u)e^{-k\pi i u}du=e^{-\frac{\sigma^2(k\pi)^2}{2}}.
\]
This is the Fourier transform on $\mathbb{R}$ of a Gaussian \cite[Theorem 1.4, Chapter 5]{SteinShakarchi2003}
\begin{corollary}\label{f32} Let $D_{\sigma}$ from Proposition \ref{f5} then
	\begin{align*}
		D[k, k]= \int_{-\infty}^{+\infty}\rho_{\sigma}(z)e^{-k\pi i z}dz=e^{-\frac{\sigma^2\pi^2k^2}{2}}.
	\end{align*}
In particular, the entries $D[k, k]$ converge to $0$ exponentially fast as $|k|$ goes to infinity.
\end{corollary}

\begin{theorem}\label{thm:decay_fourier}
Let \( \sigma > 0 \), and let \( D_{\sigma}:\ell_{\infty} \to \ell_{\infty} \) be the diagonal operator defined by
\[
D_{\sigma}[k, k] = \int_{-\infty}^{+\infty} \rho_{\sigma}(z) e^{-k\pi i z} \, dz = e^{-\frac{\sigma^2\pi^2 k^2 }{2}}.
\]
Then:
\begin{enumerate}
    \item The image of \( D_{\sigma} \) lies in the space of exponentially decaying sequences $\ell^{\exp}_{\sigma}$.
    \item In particular, for any \( f \in L^1([-1,1]) \), the Fourier coefficients of \( N_\sigma f \) satisfy
    \[
    \mathcal{F}(N_\sigma f) = D_\sigma \mathcal{F}(f) \in \ell^{\exp}_{\sigma},
    \]
    i.e., \( N_\sigma f \) has exponentially decaying Fourier coefficients.
\end{enumerate}
\end{theorem}
\begin{proof}
Let \( (b_k)_{k \in \mathbb{Z}} \in \ell_{\infty} \), so that \( |b_k| \leq C \) for all \( k \). Then
\[
D_{\sigma}(b_k) = (e^{-\frac{\sigma^2 \pi^2 k^2}{2}} b_k)_{k \in \mathbb{Z}}.
\]
Since \( e^{-\frac{\sigma^2 \pi^2 k^2}{2}} \leq e^{-\frac{\sigma^2 |k|}{2}} \), we have
\[
|D_{\sigma}(b_k)| \leq C e^{-\frac{\sigma^2 |k|}{2}},
\]
so \( D_{\sigma}(b_k) \in \ell^{\exp}_{\sigma} \), as claimed.
\end{proof}

\begin{corollary}\label{cor:analytic}
The stationary density \( f_{\sigma} \) is unique and real-analytic.
\end{corollary}

\begin{proof}
By Proposition~\ref{prop:unique} the annealed Perron–Frobenius operator \( P_{\sigma} \) has a unique fixed point \( f_{\sigma} \) in \( L^1([{-1},1]) \), corresponding to the unique stationary density of the random dynamical system.

Moreover, since \( f_{\sigma} = P_{\sigma} f_{\sigma} = N_{\sigma} P f_{\sigma} \), and \( P f_{\sigma} \in L^1 \), Theorem~\ref{thm:decay_fourier} implies that \( f_{\sigma} \) has exponentially decaying Fourier coefficients. Hence, \( f_{\sigma} \in \ell^{\exp}_\sigma \), and its Fourier series converges uniformly to a real-analytic function.

Therefore, \( f_{\sigma} \) is real-analytic on \( [-1,1] \).
\end{proof}

\section{Fourier Approximation}\label{c23}

In this section, we show how to approximate the fixed point of the annealed Perron-Frobenius operator \( P_\sigma \) using the Fourier basis.

Since \( P_\sigma \) acts on functions, our goal is to approximate its fixed point—a function—via a finite-dimensional computation. To do this, we project \( P_\sigma \) onto a finite-dimensional subspace spanned by the first \( 2k + 1 \) Fourier modes. This allows us to work with a truncated version of the operator and compute a numerical approximation of the fixed point in coefficient space. However, note that recovering the full analytic function from its truncated Fourier representation is not always possible.

To implement this approximation, we define a projection operator \( \pi_k \) that truncates the Fourier expansion to the first \( 2k + 1 \) terms.

\begin{definition}
Let \( \pi_k : L^1([-1,1]) \to L^2([-1,1]) \) be defined by
\[
\pi_k f(x) = \sum_{j=-k}^{k} \left( \frac{1}{2} \int_{-1}^{1} f(x) e^{-j \pi i x} dx \right) e^{j \pi i x} = \sum_{j=-k}^{k} \mathcal{F}(f)[j] e^{j \pi i x}.
\]
This is the \textbf{Galerkin projection} onto the Fourier modes of order at most \( k \).
\end{definition}

To handle truncation directly in Fourier coefficient space, we define the projection \( \widetilde{\pi}_k \) on sequences:

\begin{definition}
Let \( \widetilde{\pi}_k : \ell_\infty \to \ell_\infty \) be defined by
\[
\widetilde{\pi}_k((a_j)_{j \in \mathbb{Z}}) = (b_j)_{j \in \mathbb{Z}}, \quad \text{where } b_j = 
\begin{cases}
a_j, & |j| \leq k, \\
0,   & |j| > k.
\end{cases}
\]
\end{definition}

The following lemma summarizes key algebraic properties of \( \pi_k \), \( \widetilde{\pi}_k \), and their relation to the diagonal operator \( D_\sigma \) defined earlier.

\begin{lemma} \label{f42}
Let \( \pi_k \), \( \widetilde{\pi}_k \), and \( D_\sigma \) be as above. Then:
\begin{enumerate}
    \item \( \pi_k^n = \pi_k \) and \( \widetilde{\pi}_k^n = \widetilde{\pi}_k \) for all \( n \in \mathbb{N} \) (they are projections).
    \item If \( k < k' \), then \( \pi_k \pi_{k'} = \pi_k \).
    \item \( D_\sigma \) commutes with both \( \widetilde{\pi}_k \) and \( 1 - \widetilde{\pi}_k \).
    \item The following identity holds:
    \[
    \mathcal{F}^{-1} \widetilde{\pi}_k \mathcal{F} = \pi_k.
    \]
\end{enumerate}
\end{lemma}

Recall that \( D_\sigma \mathcal{F} = \mathcal{F} N_\sigma \) and that the diagonal entries of \( D_\sigma \) decay exponentially. We now quantify how close \( N_\sigma \) is to its truncated version.

\begin{theorem}\label{thm:tail_estimate}
Let $\pi_k$ denote the Fourier projection onto modes $|j|\le k$ on $[-1,1]$, and let $N_\sigma$ be the
Gaussian noise operator. Set
\[
\Gamma_{\sigma,k}
:=\sqrt{\frac{1}{\sigma\sqrt{\pi}}\coth\!\Bigl(\frac{1}{2\sigma^2}\Bigr)}\;
e^{-\frac{\sigma^2\pi^2 k^2}{2}}.
\]
Then
\[
\|N_\sigma (1-\pi_k)\|_{L^1\to L^2}\le \Gamma_{\sigma,k},
\qquad
\|(1-\pi_k)N_\sigma\|_{L^1\to L^2}\le \Gamma_{\sigma,k}.
\]
\end{theorem}

\begin{proof}
From Parseval's identity, we have
\begin{equation}\label{eq:parseval}
\|g\|_{L^2([-1,1])}^2 = 2\sum_{j\in\mathbb Z}|\mathcal{F}(g)[j]|^2.
\end{equation}
Moreover, $N_\sigma$ is a Fourier multiplier:
\[
\mathcal{F}(N_\sigma f)[j]=D[j,j]\,\mathcal{F}(f)[j],
\]
and $\pi_k$ is also a Fourier multiplier. Hence $N_\sigma$ and $\pi_k$ commute, and
$N_\sigma(1-\pi_k)=(1-\pi_k)N_\sigma$, so it suffices to bound one of them.

Let $f\in L^1([-1,1])$. Since $(1-\pi_k)$ removes modes $|j|\le k$,
\[
\|N_\sigma(1-\pi_k)f\|_{L^2}^2
=2\sum_{|j|>k}|D[j,j]|^2\,|\mathcal{F}(f)[j]|^2.
\]
Using $|\mathcal{F} (f)[j]|\le \frac12\|f\|_{L^1}$ (Lemma~\ref{lem:fourier}(1)), we obtain
\[
\|N_\sigma(1-\pi_k)f\|_{L^2}^2
\le 2\sum_{|j|>k} e^{-\sigma^2\pi^2 j^2}\,\frac14\|f\|_{L^1}^2
=\frac12\Bigl(\sum_{|j|>k} e^{-\sigma^2\pi^2 j^2}\Bigr)\|f\|_{L^1}^2.
\]
To bound the tail sum, write $j=k+m$ with $m\ge1$ and note $(k+m)^2\ge k^2+m^2$, hence
\[
\sum_{j\ge k+1} e^{-\sigma^2\pi^2 j^2}
\le e^{-\sigma^2\pi^2 k^2}\sum_{m\ge1} e^{-\sigma^2\pi^2 m^2}
\le e^{-\sigma^2\pi^2 k^2}\sum_{m\in\mathbb Z} e^{-\sigma^2\pi^2 m^2}.
\]
Therefore
\[
\sum_{|j|>k} e^{-\sigma^2\pi^2 j^2}
\le 2e^{-\sigma^2\pi^2 k^2}\sum_{m\in\mathbb Z} e^{-\sigma^2\pi^2 m^2}.
\]
By Jacobi/Poisson summation,
\[
\sum_{m\in\mathbb Z} e^{-\sigma^2\pi^2 m^2}
=\frac{1}{\sigma\sqrt{\pi}}\sum_{n\in\mathbb Z} e^{-n^2/\sigma^2}
\le \frac{1}{\sigma\sqrt{\pi}}\coth\!\Bigl(\frac{1}{2\sigma^2}\Bigr),
\]
using the bound from Remark~\ref{rem:coth-bound}. Combining the previous estimates gives
\[
\|N_\sigma(1-\pi_k)f\|_{L^2}^2
\le \frac12\cdot 2e^{-\sigma^2\pi^2 k^2}\cdot
\frac{1}{\sigma\sqrt{\pi}}\coth\!\Bigl(\frac{1}{2\sigma^2}\Bigr)\,\|f\|_{L^1}^2,
\]
and taking square roots yields the stated bound.
\end{proof}

\begin{remark}
The factor $\coth\!\bigl(\frac{1}{2\sigma^2}\bigr)$ controls the wrap--around effect introduced by periodizing the Gaussian kernel (cf.\ Lemma~\ref{c54} and Remark~\ref{rem:coth-bound}).
\end{remark}

Next, we show that the projected discretization of \( P_\sigma \) coincides with that of the individual components:

\begin{lemma}\label{lem:proj_commute}
The following identity holds:
\[
\pi_k N_\sigma \pi_k P \pi_k = \pi_k P_\sigma \pi_k.
\]
\end{lemma}

\begin{proof}
Follows directly from the identities \( P_\sigma = N_\sigma P \), \( \pi_k^2 = \pi_k \) and $\pi_k N_{\sigma}=N_{\sigma}\pi_k$; the latter follows from Proposition \ref{f5} and the commutation relations between \( D_\sigma \), \( \widetilde{\pi}_k \), and \( \mathcal{F} \) (see Lemma~\ref{f42}).
\end{proof}

While \( N_\sigma \) is a contraction on \( L^1 \), the discretized version \( \pi_k N_\sigma \) may, in principle, increase the norm. We now bound this growth:

\begin{lemma}\label{lem:tail_L1_Linfty}
Let $N_\sigma$ be the Gaussian noise operator with periodic boundary condition, and let $\pi_k$ be the Fourier
projection onto modes $|j|\le k$.
Then for every $f\in L^1([-1,1])$,
\begin{align}
\|(1-\pi_k)N_\sigma f\|_{L^\infty([-1,1])}
&\le \frac12\|f\|_{L^1([-1,1])}\sum_{|j|>k} e^{-\frac{\sigma^2\pi^2 j^2}{2}},
\label{eq:tail_Linfty}\\
\|(1-\pi_k)N_\sigma f\|_{L^1([-1,1])}
&\le \|f\|_{L^1([-1,1])}\sum_{|j|>k} e^{-\frac{\sigma^2\pi^2 j^2}{2}}.
\label{eq:tail_L1}
\end{align}
In particular,
\[
\|(1-\pi_k)N_\sigma\|_{L^1\to L^1}\le \Gamma^{(1)}_{\sigma,k},
\qquad
\Gamma^{(1)}_{\sigma,k}:=\sum_{|j|>k} e^{-\frac{\sigma^2\pi^2 j^2}{2}}.
\]
Moreover, letting $a:=\sigma^2\pi^2/2$, the tail admits the explicit bound
\begin{equation}\label{eq:Gamma1_explicit}
\Gamma^{(1)}_{\sigma,k}\le \frac{2}{\sigma^2\pi^2\,k}e^{-\frac{\sigma^2\pi^2 k^2}{2}}
\qquad (k\ge 1).
\end{equation}
\end{lemma}

\begin{proof}
Since $N_\sigma$ is a Fourier multiplier,
\[
\mathcal{F}(N_\sigma f)[j]=D[j,j]\mathcal{F}(f)[j],
\]
Therefore
\[
(1-\pi_k)N_\sigma f(x)=\sum_{|j|>k} D[j,j]\mathcal{F}(f)[j]\,e^{j\pi i x}.
\]
Using $|\mathcal{F}(f)[j]|\le \frac12\|f\|_{L^1}$, we obtain for each $x\in[-1,1]$,
\[
|(1-\pi_k)N_\sigma f(x)|
\le \sum_{|j|>k} D[j,j]|\mathcal{F}(f)[j]|
\le \frac12\|f\|_{L^1}\sum_{|j|>k} e^{-\frac{\sigma^2\pi^2 j^2}{2}},
\]
which yields \eqref{eq:tail_Linfty}. Since $m([-1,1])=2$, we also have
$\|g\|_{L^1([-1,1])}\le 2\|g\|_{L^\infty([-1,1])}$, hence \eqref{eq:tail_L1}.

For \eqref{eq:Gamma1_explicit}, write $\Gamma^{(1)}_{\sigma,k}=\displaystyle 2\sum_{j\ge k+1}e^{-a j^2}$ and bound the tail by an integral:
\[
\sum_{j\ge k+1}e^{-a j^2}\le \int_k^\infty e^{-a x^2}\,dx
\le \frac{e^{-a k^2}}{2ak},
\]
which gives $\Gamma^{(1)}_{\sigma,k}\le \frac{1}{ak}e^{-a k^2}
=\frac{2}{\sigma^2\pi^2 k}e^{-\frac{\sigma^2\pi^2 k^2}{2} }$.
\end{proof}

\begin{corollary}\label{cor:discretized_norm}
With $\Gamma^{(1)}_{\sigma,k}$ as in Lemma~\ref{lem:tail_L1_Linfty}, one has
\[
\|\pi_k N_\sigma\|_{L^1\to L^1}\le 1+\Gamma^{(1)}_{\sigma,k}.
\]
\end{corollary}

\begin{proof}
Since $\pi_k=\textrm{Id}-(\textrm{Id}-\pi_k)$ and $\|N_\sigma\|_{L^1\to L^1}\le 1$ (Lemma~\ref{lem:regN}),
\[
\|\pi_k N_\sigma\|_{L^1\to L^1}
\le \|N_\sigma\|_{L^1\to L^1}+\|(1-\pi_k)N_\sigma\|_{L^1\to L^1}
\le 1+\Gamma^{(1)}_{\sigma,k}.
\]
\end{proof}

\begin{remark}
Together, these results justify the use of the finite-rank approximation \( \pi_k P_\sigma \pi_k \) as a computable and controlled approximation of \( P_\sigma \), with a well-quantified truncation error.
\end{remark}
Next, we compare the \(L^1\) norm of the fixed point of \(P_{\sigma}\) with the norm of its Galerkin projection.

\begin{lemma}\label{lemma:L1discr}
Let \(f_{\sigma}\) be the fixed point of \(P_{\sigma}\). Then
\[
\|\pi_k f_{\sigma}\|_{L^1} = \|\pi_k N_{\sigma} P f_{\sigma}\|_{L^1} \leq (1 + \Gamma^{(1)}_{\sigma,k}) \|f_{\sigma}\|_{L^1}.
\]
\end{lemma}

We now define the discretized version of \(P_{\sigma}\), which we will use in our approximation scheme.

\begin{definition}
The discretized annealed Perron-Frobenius operator is defined as
\[
P_{\sigma,k} := \pi_k N_{\sigma} \pi_k P \pi_k = \pi_k N_{\sigma} P \pi_k = \pi_k P_{\sigma} \pi_k : L^1 \to L^2.
\]
\end{definition}

Note that if \(f \in \mathcal{U}_0\), the space of average $0$ functions in \(L^2([-1,1])\), then
\[
\mathcal{F}(f)[0] = \frac{1}{2} \int_{-1}^{1} f(x) dx = 0.
\]

We now state our main theorem, a variation of Theorem 3.4 from \cite{GMNP}, adapted to our context. It provides a rigorous bound on the error between the true fixed point \(f_\sigma\) and an approximate fixed point obtained from a finite-dimensional Galerkin scheme.

\begin{theorem}\label{c27}
Let \(f_{\sigma}\) be the unique fixed point of \(P_{\sigma}\), and let \(P_{\sigma,k}\) be its Galerkin approximation. 
For $i\in \mathbb{N}_0$, set $C_i:=\|P_{\sigma,k}^i|_{\mathcal{U}_0}\|_{L^2 \to L^2}$, and suppose that $C_n<1$ for some $n\in \mathbb{N}$. If \(g\) be a trigonometric polynomial of degree at most \(k\) such that \(\|P_{\sigma,k} g - g\|_{L^2} \leq \epsilon\). Then
\[
\|f_{\sigma} - g\|_{L^2} \leq \frac{1}{1 - C_n} \sum_{i=0}^{n-1} C_i \left( (1 + \Gamma_{\sigma,k} + \|\rho_{\sigma}\|_{L^2(\mathbb{R})}) \Gamma^{(1)}_{\sigma,k} \|f_{\sigma}\|_{L^1} + \epsilon \right).
\]
\begin{align*}
	\|f_\sigma - g\|_{L^2} \leq \frac{1}{1 - C_n} \sum_{i = 0}^{n - 1} C_i \left( \delta+ \epsilon \right). 
\end{align*}
where \[
\delta 
 \leq \left(\Gamma_{\sigma,k}(1+\Gamma_{\sigma,k}^{(1)})+ \|\rho_\sigma\|_{L^2(\mathbb{R})} \sqrt{ \coth\!\Bigl(\frac{1}{2\sigma^2}\Bigr)}\cdot \Gamma^{(1)}_{\sigma,k} \right) \|f_{\sigma}\|_{L^1}
\]
\end{theorem}

\begin{remark}\label{rem:symetrized}
All constants involved in the theorem can be rigorously computed or estimated with the aid of a computer. In particular, the contraction bound \(C_n\), the intermediate norms \(C_i\), and the residual \(\|P_{\sigma,k} g - g\|_{L^2}\) are all computable. The choice of the approximation \(g\) is crucial: we typically use a numerically computed fixed point of \(P_{\sigma,k}\),  and enforce symmetry in its truncated Fourier representation to ensure that the resulting trigonometric polynomial is real. An approximation $g$ choice in this way will be referred to as a \textbf{symmetrized} trigonometric polynomial approximating $f_{\sigma}$
\end{remark}

To prove Theorem~\ref{c27}, we first recall two lemmas adapted from \cite{GMNP}, which control the behavior of powers of the discretized operator \(P_{\sigma,k}\) restricted to the subspace \(\mathcal{U}_0\).

\begin{lemma}[\cite{GMNP}]\label{c8}
Let \(C_i := \|P_{\sigma,k }^{i}|_{\mathcal{U}_0}\|_{L^{2}\rightarrow L^{2}}\) for \(i \in \mathbb{N}_0\), and suppose that \(C_n < 1\) for some \(n \in \mathbb{N}\). Then:
\begin{enumerate}
    \item \(\sum_{i=0}^{\infty} C_i \leq \frac{1}{1 - C_n} \sum_{i=0}^{n-1} C_i\),
    \item There exist constants \(C > 0\) and \(\theta \in (0,1)\) such that \(\|P_{\sigma,k}^i|_{\mathcal{U}_0}\|_{L^2 \to L^2} \leq C \theta^i\) for all \(i\).
\end{enumerate}
\end{lemma}

\begin{proof}
\textbf{(1)} Write \(i = qn + r\) with \(0 \leq r < n\). Using the invariance \(P_{\sigma,k}(\mathcal{U}_0) \subset \mathcal{U}_0\), we get:
\[
\|P_{\sigma,k}^i|_{\mathcal{U}_0}\|_{L^2 \to L^2} \leq \|P_{\sigma,k}^n\|_{L^2\rightarrow L^2}^q \cdot \|P_{\sigma,k}^r\|_{L^2\rightarrow L^2} \leq C_n^q C_r.
\]
Summing over \(i\), we reorganize by blocks of length \(n\):
\[
\sum_{i=0}^\infty C_i \leq \sum_{q=0}^\infty C_n^q \sum_{r=0}^{n-1} C_r = \left( \sum_{q=0}^\infty C_n^q \right) \sum_{r=0}^{n-1} C_r = \frac{1}{1 - C_n} \sum_{i=0}^{n-1} C_i.
\]

\textbf{(2)} From the previous bound,
\[
\|P_{\sigma,k}^i\|_{L^2\rightarrow L^2} \leq C_n^{\lfloor i/n \rfloor} \cdot \max_{r < n} C_r \leq C_n^{i/n - 1} \cdot \max_{r < n} C_r,
\]
so the desired exponential decay holds with \(\theta = C_n^{1/n}\) and \(C = \max_{r < n} C_r / C_n\).
\end{proof}

We now relate the size of a vector in \(\mathcal{U}_0\) to the residual of its image under \(P_{\sigma,k}\).

\begin{lemma}\label{c9}
Let \(C_i := \|P_{\sigma,k }^{i}|_{\mathcal{U}_0}\|_{L^2 \to L^2}\) with \(C_n < 1\) for some \(n \in \mathbb{N}\). Then for any \(v \in \mathcal{U}_0\),
\[
\|v\|_{L^2} \leq \frac{1}{1 - C_n} \sum_{i=0}^{n-1} C_i \|P_{\sigma,k} v - v\|_{L^2}.
\]
\end{lemma}

\begin{proof}
We telescope:
\[
v = (P_{\sigma,k} v - v) + (P_{\sigma,k}^2 v - P_{\sigma,k} v) + \cdots + (P_{\sigma,k}^n v - P_{\sigma,k}^{n-1} v) - P_{\sigma,k}^n v.
\]
Taking \(L^2\) norms:
\[
\|v\|_{L^2} \leq \sum_{i=0}^{n-1} \|P_{\sigma,k}^i (P_{\sigma,k} v - v)\|_{L^2} + \|P_{\sigma,k}^n v\|_{L^2}.
\]
By Lemma~\ref{c8}(2), \(\|P_{\sigma,k}^n v\|_{L^2} \to 0\) as \(n \to \infty\). So:
\[
\|v\|_{L^2} \leq \sum_{i=0}^\infty C_i \|P_{\sigma,k} v - v\|_{L^2} \leq \frac{1}{1 - C_n} \sum_{i=0}^{n-1} C_i \|P_{\sigma,k} v - v\|_{L^2},
\]
as claimed.
\end{proof}

\begin{remark}\label{rem:coth-bound}
For $\sigma>0$ we have the elementary estimate $n^2\ge n$ for all $n\ge1$, hence
\[
\sum_{n\ge1} e^{-n^2/\sigma^2}
\le \sum_{n\ge1} e^{-n/\sigma^2}
=\frac{e^{-1/\sigma^2}}{1-e^{-1/\sigma^2}}.
\]
Therefore
\[
1+2\sum_{n\ge1} e^{-n^2/\sigma^2}
\le 1+\frac{2e^{-1/\sigma^2}}{1-e^{-1/\sigma^2}}
=\frac{1+e^{-1/\sigma^2}}{1-e^{-1/\sigma^2}}
=\coth\!\Bigl(\frac{1}{2\sigma^2}\Bigr).
\]
\end{remark}

\begin{lemma}\label{c54}
Let $\rho_\sigma(x)=\frac{1}{\sigma\sqrt{2\pi}}e^{-x^2/(2\sigma^2)}$ and let
\[
(\tau_*\rho_\sigma)(x):=\sum_{j\in\mathbb Z}\rho_\sigma(x+2j),\qquad x\in[-1,1],
\]
be its $2$--periodization. Then $\tau_*\rho_\sigma\in L^2([-1,1])$ and
\begin{equation}\label{eq:tau_rho_L2_parseval}
\|\tau_*\rho_\sigma\|_{L^2([-1,1])}^2
= \frac12\sum_{k\in\mathbb Z} e^{-\sigma^2\pi^2 k^2}.
\end{equation}
Moreover,
\begin{equation}\label{eq:rho_L2_R}
\|\rho_\sigma\|_{L^2(\mathbb R)}^2=\frac{1}{2\sigma\sqrt{\pi}},
\end{equation}
and, since $\rho_\sigma>0$ everywhere,
\begin{equation}\label{eq:strict_ineq}
\|\tau_*\rho_\sigma\|_{L^2([-1,1])}^2>\|\rho_\sigma\|_{L^2(\mathbb R)}^2.
\end{equation}
Equivalently, using the Poisson summation identity,
\begin{equation}\label{eq:tau_rho_L2_theta}
\|\tau_*\rho_\sigma\|_{L^2([-1,1])}^2
=\frac{1}{2\sigma\sqrt{\pi}}\sum_{n\in\mathbb Z} e^{-n^2/\sigma^2}
\leq \|\rho_\sigma\|_{L^2(\mathbb R)}^2 \coth\!\Bigl(\frac{1}{2\sigma^2}\Bigr).
\end{equation}
\end{lemma}

\begin{proof}
For $f=\tau_*\rho_\sigma$, we compute for each $k\in\mathbb Z$:
\begin{align*}
\mathcal{F}(\tau_*\rho_\sigma)[k]
&=\frac12\int_{-1}^1\sum_{j\in\mathbb Z}\rho_\sigma(x+2j)\,e^{-k\pi i x}\,dx \\
&=\frac12\sum_{j\in\mathbb Z}\int_{-1}^1\rho_\sigma(x+2j)\,e^{-k\pi i x}\,dx
=\frac12\int_{\mathbb R}\rho_\sigma(u)\,e^{-k\pi i u}\,du,
\end{align*}
where we used Tonelli/Fubini (absolute convergence) and the change of variables $u=x+2j$.
Since the Fourier transform of a Gaussian is Gaussian,
\[
\int_{\mathbb R}\rho_\sigma(u)\,e^{-k\pi i u}\,du
= e^{-\frac{\sigma^2\pi^2 k^2}{2} },
\]
hence
\[
\mathcal{F}(\tau_*\rho_\sigma)[k]=\frac12 e^{-\frac{\sigma^2\pi^2 k^2}{2} }.
\]
Plugging this into \eqref{eq:parseval} yields
\[
\|\tau_*\rho_\sigma\|_{L^2([-1,1])}^2
=2\sum_{k\in\mathbb Z}\Bigl|\frac12 e^{-\frac{\sigma^2\pi^2 k^2}{2} }\Bigr|^2
=\frac12\sum_{k\in\mathbb Z}e^{-\sigma^2\pi^2 k^2},
\]
which is \eqref{eq:tau_rho_L2_parseval}.

For \eqref{eq:rho_L2_R}, we compute
\[
\|\rho_\sigma\|_{L^2(\mathbb R)}^2
=\int_{\mathbb R}\Bigl(\frac{1}{\sigma\sqrt{2\pi}}e^{-x^2/(2\sigma^2)}\Bigr)^2\,dx
=\frac{1}{2\pi\sigma^2}\int_{\mathbb R}e^{-x^2/\sigma^2}\,dx
=\frac{1}{2\sigma\sqrt{\pi}}.
\]

Finally, \eqref{eq:strict_ineq} follows immediately from \eqref{eq:tau_rho_L2_parseval} and Jabobi/Poisson summation:
\begin{align*}
\|\tau_*\rho_\sigma\|_{L^2([-1,1])}^2
=	\frac12\sum_{k\in\mathbb Z} e^{-\sigma^2\pi^2 k^2}=\frac{1}{2}\left(\frac{1}{\sigma \sqrt{\pi}} \displaystyle \sum_{n\in\mathbb{Z}}e^{-n^2/\sigma^2} \right)>\frac{1}{2\sigma\sqrt{\pi}}=\|\rho_{\sigma}\|^2_{L^2(\mathbb{R})}
\end{align*}
i.e.\ the extra $n\neq 0$ terms are strictly positive.
The alternative expression \eqref{eq:tau_rho_L2_theta} follows from the Poisson summation formula applied to
$\sum_{k\in\mathbb Z}e^{-\sigma^2\pi^2 k^2}$.
\end{proof}

\begin{proof}[Proof of Theorem \ref{c27}]
Let \( v = f_\sigma - g \). Then
\[
\|v\|_{L^2} \leq \frac{1}{1 - C_n} \sum_{i = 0}^{n - 1} C_i \|P_{\sigma,k} v - v\|_{L^2}, \quad \text{by Lemma~\ref{c9}}.
\]
Using linearity and triangle inequality:
\[
\|P_{\sigma,k} v - v\|_{L^2} \leq \|P_{\sigma,k} f_\sigma - f_\sigma\|_{L^2} + \|P_{\sigma,k} g - g\|_{L^2} = \delta + \epsilon,
\]
where \( \epsilon = \|P_{\sigma,k} g - g\|_{L^2} \). It remains to bound \( \delta = \|P_{\sigma,k} f_\sigma - f_\sigma\|_{L^2} \).

Using Lemma~\ref{lem:proj_commute}, we write
\[
P_{\sigma,k} f_\sigma - f_\sigma = (\pi_k N_\sigma \pi_k P \pi_k - N_\sigma P) f_\sigma = ( \pi_k -1) N_\sigma P \pi_k f_\sigma + N_\sigma P (\pi_k-1) f_\sigma.
\]

For the first term, by Theorem~\ref{thm:tail_estimate} and Lemma~\ref{lemma:L1discr}:
\[
\|(1 - \pi_k) N_\sigma P \pi_k f_\sigma\|_{L^2} \leq \Gamma_{\sigma,k} \|\pi_k f_\sigma\|_{L^1} \leq \Gamma_{\sigma,k}(1 + \Gamma^{(1)}_{\sigma,k}) \|f_\sigma\|_{L^1}.
\]

For the second term, using Lemma~\ref{c54}, Remark~\ref{rem:coth-bound} Lemma~\ref{lem:tail_L1_Linfty} and Young’s inequality:
\[
\|N_\sigma P (1 - \pi_k) f_\sigma\|_{L^2} \leq \|\rho_\sigma\|_{L^2(\mathbb{R})} \sqrt{ \coth\!\Bigl(\frac{1}{2\sigma^2}\Bigr)}\cdot \Gamma^{(1)}_{\sigma,k} \|f_\sigma\|_{L^1}.
\]

Combining both:
\[
\delta = \|P_{\sigma,k} f_\sigma - f_\sigma\|_{L^2} \leq \left(\Gamma_{\sigma,k}(1+\Gamma_{\sigma,k}^{(1)})+ \|\rho_\sigma\|_{L^2(\mathbb{R})} \sqrt{ \coth\!\Bigl(\frac{1}{2\sigma^2}\Bigr)}\cdot \Gamma^{(1)}_{\sigma,k} \right) \|f_{\sigma}\|_{L^1}
\]

Plugging back:
\begin{align*}
\|f_\sigma - g\|_{L^2} \leq \frac{1}{1 - C_n} \sum_{i = 0}^{n - 1} C_i \left( \delta+ \epsilon \right). 
\end{align*}
\end{proof}

\subsection{Implementation of the discretization for the discretized operator}

To discretize the deterministic transfer operator $P$ we use a Fourier--Galerkin projection.
Let $e_k(x):=e^{k\pi i x}$ and define the matrix coefficients by
\[
P_{k,\ell}:=\frac12\int_{-1}^1 \overline{e_k(x)}\,(P e_\ell)(x)\,dx.
\]
Using the duality relation $\int \varphi\cdot P\psi\,dx=\int (\varphi\circ T)\cdot \psi\,dx$ we obtain
\[
P_{k,\ell}
=\frac12\int_{-1}^1 \overline{e_k(T(x))}\,e_\ell(x)\,dx.
\]
For $\alpha>2$, the map $T_{\alpha,\beta}$ is $C^2$ on $(-1,1)$, hence the function
\[
f_k(x):=\overline{e_k(T_{\alpha,\beta}(x))}
\]
is $C^2$ on $(-1,1)$. Moreover, while $f_k(1)=f_k(-1)$ (since $T_{\alpha,\beta}(1)=T_{\alpha,\beta}(-1)$),
its derivative need not match at the endpoints; in other words, the $2$--periodic extension of $f_k$ typically
has a jump in the first derivative at $\pm1$.
Consequently, the Fourier coefficients of $f_k$ (and hence the entries $P_{k,\ell}$) decay quadratically:
for every $\ell\neq 0$,
\begin{equation}\label{eq:Pkell_decay}
|P_{k,\ell}|
\;=\;
\bigl|\mathcal{F}(f_k)[\ell]\bigr|
\;\le\;
\frac{1}{2(\pi \ell)^2}\Bigl(|f_k'(1)-f_k'(-1)|+\|f_k''\|_{L^1}\Bigr),
\end{equation}
and in particular $|P_{k,\ell}|\le C_k/|\ell|^2$ with an explicit constant $C_k$.

For performance we evaluate the Fourier coefficients from point samples using \texttt{FFTW}.
Concretely, for a $2$--periodic integrand $f$ we sample on an equispaced grid
\[
x_m:=-1+\frac{2m}{N},\qquad m=0,\dots,N-1,
\]
and compute the discrete Fourier transform of the vector $(f(x_m))_{m=0}^{N-1}$ by \texttt{FFTW}.
This yields the discrete coefficients
\[
\mathcal{F}_N(f)[\ell]
:=
\frac1N\sum_{m=0}^{N-1} f(x_m)\,e^{-\pi i \ell x_m},
\qquad |\ell|\le \frac N2,
\]
which approximate the true Fourier coefficients
\[
\mathcal{F}(f)[\ell]=\frac12\int_{-1}^1 f(x)\,e^{-\pi i \ell x}\,dx.
\]
The discretization error is an \emph{aliasing} (wrap-around) effect: one has the exact identity
\begin{equation}\label{eq:aliasing}
\mathcal{F}_N(f)[\ell]=\sum_{q\in\mathbb Z} \mathcal{F}(f)[\ell+qN],
\end{equation}
and hence
\begin{equation}\label{eq:aliasing_err}
|\mathcal{F}_N(f)[\ell]-\mathcal{F}(f)[\ell]|
\le
\sum_{q\neq 0}|\mathcal{F}(f)[\ell+qN]|.
\end{equation}
In our application $f=f_k:=\overline{e_k\circ T}$ is $C^2$ on $(-1,1)$ and its $2$--periodic extension typically has a jump in the first derivative at $\pm 1$; consequently
$|\mathcal{F}(f_k)[n]|\le C_k/|n|^2$ for $n\neq 0$ with explicit $C_k$ (see \eqref{eq:Pkell_decay}).
Combining this decay with \eqref{eq:aliasing_err} gives the oversampling bound
\begin{equation}\label{eq:aliasing_O_N2}
|\mathcal{F}_N(f_k)[\ell]-\mathcal{F}(f_k)[\ell]|
\le
\frac{\pi^2}{3}\,\frac{C_k}{N^2},
\qquad |\ell|\le \frac N2.
\end{equation}
We therefore choose an oversampling factor $s>1$ and take $N=2^n$ (so that we are in the radix-$2$ setting of Lemma~\ref{lem:fft-roundoff} and of the roundoff analysis in~\cite{BrisebarreEtAl2020}), with
$N\ge s(2K+1)$, compute $\mathcal{F}_N(f_k)$ by FFT, and retain the modes $|\ell|\le K$.

We refer to \cite{BrisebarreEtAl2020} for a detailed modern analysis of rounding errors in radix-$2$
Cooley--Tukey FFTs. As already emphasized by \cite{Higham1996}, the comparatively low operation
count of the FFT (only $O(N\log N)$ floating-point operations) yields normwise error bounds with
small, explicitly controlled constants.

\begin{lemma}\label{lem:fft-roundoff}
Let $x\in\mathbb{C}^N$ and let
\[
\mathcal{F}_N := N^{-1/2}\bigl(\omega_N^{jk}\bigr)_{j,k=0}^{N-1},\qquad \omega_N:=e^{-2\pi i/N},
\]
be the \emph{unitary} Fourier matrix. Denote by $y:=\mathcal{F}_Nx$ the exact transform and by $\widehat y$
the result returned by an FFT implementation in floating-point arithmetic with unit roundoff $u$.

The FFT is said to be \emph{normwise backward stable} if there exists a constant $k_N>0$ such that,
whenever $k_Nu<1$, one can represent
\[
\widehat y = \mathcal{F}_N(x+\Delta x)
\qquad\text{with}\qquad
\|\Delta x\|_2 \le \gamma_{k_N}\,\|x\|_2,
\qquad
\gamma_{k_N}:=\frac{k_Nu}{1-k_Nu}.
\]
Since $\mathcal{F}_N$ is unitary, the same bound holds in forward form:
\[
\|\widehat y-y\|_2 \le \gamma_{k_N}\,\|x\|_2 .
\]
For the radix-$2$ Cooley--Tukey FFT ($N=2^n$), one may take $k_N=O(\log_2 N)$; more precisely,
if the twiddle factors are precomputed with absolute error bounded by $c_Nu$, then a valid choice is
\[
k_N = \bigl(\sqrt2\,c_N + 4 + \sqrt2\bigr)\log_2 N ,
\]
see, e.g., \cite{BrisebarreEtAl2020}.
\end{lemma}

\begin{corollary}[Coefficient-wise enclosure from an $L^2$ bound]\label{cor:fft-coeff-enclosure}
Let $x\in\mathbb{C}^N$ be the vector of sampled values of a $2$--periodic function $f$ on the uniform grid,
let $\mathcal{F}_N(f)\in\mathbb{C}^N$ denote the exact discrete Fourier transform in the unitary normalization above,
and let $\widetilde{\mathcal{F}}_N(f)$ be the output returned by \texttt{FFTW} after applying the corresponding
$N^{-1/2}$ scaling.
If $\|\widetilde{\mathcal{F}}_N(f)-\mathcal{F}_N(f)\|_2\le \varepsilon_2$, then for each mode $\ell$,
\[
\bigl|\widetilde{\mathcal{F}}_N(f)[\ell]-\mathcal{F}_N(f)[\ell]\bigr|
\le \|\widetilde{\mathcal{F}}_N(f)-\mathcal{F}_N(f)\|_\infty
\le \|\widetilde{\mathcal{F}}_N(f)-\mathcal{F}_N(f)\|_2
\le \varepsilon_2,
\]
and hence
\[
\mathcal{F}_N(f)[\ell]\in \bigl[\widetilde{\mathcal{F}}_N(f)[\ell]-\varepsilon_2,\ \widetilde{\mathcal{F}}_N(f)[\ell]+\varepsilon_2\bigr].
\]
In particular, under the assumptions of Lemma~\ref{lem:fft-roundoff} one may take
$\varepsilon_2:=\gamma_{k_N}\,\|x\|_2$ (with $\gamma_{k_N}=k_Nu/(1-k_Nu)$).
\end{corollary}
\begin{lemma}\label{lem:Pkell_total_error}
Fix $k\in\mathbb{Z}$ and set $f_k:=e_k\circ T$.
Let
\[
P_{k,\ell}:=\mathcal{F}(f_k)[\ell]
=\frac12\int_{-1}^1 f_k(x)\,e^{-\pi i \ell x}\,dx
\]
be the exact Fourier coefficient (in our normalization), and let $N\in\mathbb{N}$.
Define the exact discrete coefficient
\[
\mathcal{F}_N(f_k)[\ell]:=\frac1N\sum_{m=0}^{N-1} f_k(x_m)\,e^{-\pi i \ell x_m},
\qquad x_m:=-1+\frac{2m}{N}.
\]
Let $\widetilde{\mathcal{F}}_N(f_k)[\ell]$ be the coefficient returned by \texttt{FFTW}.

Assume:
\begin{enumerate}
\item There exists $C_k>0$ such that
$|\mathcal{F}(f_k)[n]|\le C_k/|n|^2$ for all $n\neq 0$.
\item Using interval arithmetic we compute enclosures
$X_m\ni f_k(x_m)$, and write $x_m^{\mathrm{mid}}:=\mathrm{mid}(X_m)$ and
$r_m:=\mathrm{rad}(X_m)$. Let $x^{\mathrm{mid}}\in\mathbb{C}^N$ and $r\in\mathbb{R}_+^N$ be the corresponding vectors.
\item The FFT output satisfies
$\|\widetilde{\mathcal{F}}_N(x^{\mathrm{mid}})-\mathcal{F}_N(x^{\mathrm{mid}})\|_2\le \varepsilon_2$
for an explicitly computed $\varepsilon_2$. \label{item3}
\end{enumerate}
Then, for every retained mode $|\ell|\le N/2$, 
\begin{equation}\label{eq:total_error_bound}
\bigl|P_{k,\ell}-\widetilde{\mathcal{F}}_N(f_k)[\ell]\bigr|
\;\le\;
\underbrace{\frac{\pi^2}{3}\frac{C_k}{N^2}}_{\text{aliasing}}
\;+\;
\underbrace{\frac{\|r\|_2}{\sqrt{N}}}_{\text{interval sampling}}
\;+\;
\underbrace{\varepsilon_2}_{\text{FFT roundoff}}.
\end{equation}
\end{lemma}

\begin{proof}
Decompose
\begin{align*}
|P_{k,\ell}-\widetilde{\mathcal{F}}_N(f_k)[\ell]|
\leq&
|P_{k,\ell}-\mathcal{F}_N(f_k)[\ell]|
+
|\mathcal{F}_N(f_k)[\ell]-\mathcal{F}_N(x^{\mathrm{mid}})[\ell]|\\
&+
|\mathcal{F}_N(x^{\mathrm{mid}})[\ell]-\widetilde{\mathcal{F}}_N(x^{\mathrm{mid}})[\ell]|.
\end{align*}
For $2$--periodic $f_k$ one has the exact aliasing identity
$\mathcal{F}_N(f_k)[\ell]=\sum_{q\in\mathbb{Z}}\mathcal{F}(f_k)[\ell+qN]$, hence
\[
|P_{k,\ell}-\mathcal{F}_N(f_k)[\ell]|
\le \sum_{q\neq 0}|\mathcal{F}(f_k)[\ell+qN]|
\le 2\sum_{q\ge 1}\frac{C_k}{(qN)^2}
=\frac{\pi^2}{3}\frac{C_k}{N^2}.
\]

Let $\delta_m:=f_k(x_m)-x_m^{\mathrm{mid}}$. Then $|\delta_m|\le r_m$ and
\[
|\mathcal{F}_N(f_k)[\ell]-\mathcal{F}_N(x^{\mathrm{mid}})[\ell]|
=
\Bigl|\frac1N\sum_{m=0}^{N-1}\delta_m\,e^{-\pi i\ell x_m}\Bigr|
\le \frac1N\sum_{m=0}^{N-1}|\delta_m|
\le \frac{\|\delta\|_2}{\sqrt{N}}
\le \frac{\|r\|_2}{\sqrt{N}}.
\]

By assumption (\ref{item3}) and $\|\cdot\|_\infty\le \|\cdot\|_2$,
\[
|\mathcal{F}_N(x^{\mathrm{mid}})[\ell]-\widetilde{\mathcal{F}}_N(x^{\mathrm{mid}})[\ell]|
\le \|\mathcal{F}_N(x^{\mathrm{mid}})-\widetilde{\mathcal{F}}_N(x^{\mathrm{mid}})\|_\infty
\le \varepsilon_2.
\]
Combining the three estimates yields \eqref{eq:total_error_bound}.
\end{proof}

\begin{remark}
It is worth noting that, throughout, we take $N=2^n$ and rely on \texttt{FFTW} to compute the discrete Fourier transform.
Our error bounds assume that \texttt{FFTW} uses twiddle factors computed with a precision comparable to machine precision~\cite{JohnsonFrigo2008};
accordingly, we include an a priori overestimation of $4u$, where $u$ denotes the unit roundoff.
We also assume that the machine on which the experiments were run conforms to the IEEE~754 standard for floating-point arithmetic.

In principle, one could reimplement a radix-$2$ Cooley--Tukey FFT within our codebase (and even aim at a fully verified implementation),
but absent a full formal development it is unlikely that such an implementation would match the performance, engineering maturity,
and empirical validation of the \texttt{FFTW} development team.
For this reason, we treat \texttt{FFTW} as a numerical kernel and compensate by incorporating explicit roundoff bounds
(Lemma~\ref{lem:fft-roundoff} and Corollary~\ref{cor:fft-coeff-enclosure}); see also the publicly documented \texttt{FFTW}/benchFFT
accuracy benchmarking methodology and discussion~\cite{FFTWAccuracyMethod,FFTWAccuracyComments}.
A further step towards end-to-end formal correctness would be to rely on formally generated code, for instance via verified LLVM generation
from Isabelle/HOL as in~\cite{Lammich2019}.
\end{remark}

\section{Proof of Main Result 1}\label{sec:computability}

This section establishes the abstract computability of the stationary density for the annealed system under Gaussian noise, and proves Main Result 1.
The underlying idea is that the a priori estimate on the mixing rates in Proposition \ref{prop:mixinggaussian} allows us to bound uniformly 
the mixing rates of the operators $P_{\sigma,k}$ for all $k>K$, which in turn allows us to prove that the approximation error on the stationary measure 
goes to $0$ as $k$ goes to infinity. 

\begin{main_computability}
Let $T$ be a non-singular map of the interval, and consider the random dynamical system:
\begin{align}
X_{n+1} = \tau(T(X_n) + \Omega_\sigma(n)),
\end{align}
where $\Omega_\sigma(n)$ are i.i.d. Gaussian random variables with standard deviation $\sigma$, and $\tau$ is the periodic boundary condition (Definition~\ref{c10}).

Then, the stationary density can be approximated in $L^2$, and Birkhoff averages of any observable $\phi \in L^2$ can be enclosed to arbitrary precision.
\end{main_computability}

This computability result relies crucially on the strong regularization provided by the Gaussian noise (see Proposition~\ref{prop:mixinggaussian}). Our strategy proceeds by controlling the discretization error of the annealed transfer operator $P_\sigma$, and showing convergence of the approximated fixed point $f_{\sigma, K}$ to the true fixed point $f_\sigma$ in $L^2$ norm.

\subsection{Norm Bounds for the Discretized Operator}

Recall that \\$P_{\sigma,K}:=\pi_K N_\sigma P\pi_K$ and define the $L^1$--tail constant
\[
\|(1-\pi_K)N_\sigma\|_{L^1\to L^1}
\le \Gamma^{(1)}_{\sigma,K}
\]
(cf.\ Lemma~\ref{lem:tail_L1_Linfty}).


\begin{lemma}\label{lem:L1norm}
For all $i\in\mathbb N$,
\[
\|P_{\sigma,K}^i\|_{L^1\to L^1}\le (1+\Gamma^{(1)}_{\sigma,K})^i\,(2K+1).
\]
\end{lemma}

\begin{proof}
Using $\pi_K^2=\pi_K$ (Lemma~\ref{f42}),
\[
P_{\sigma,K}^i=(\pi_K N_\sigma P)^i\,\pi_K.
\]
By Corollary~\ref{cor:discretized_norm} and $\|P\|_{L^1\to L^1}=1$,
\[
\|\pi_K N_\sigma P\|_{L^1\to L^1}\le 1+\Gamma^{(1)}_{\sigma,K}.
\]
Moreover, with $\mathcal{F}(g)[j]=\frac12\int_{-1}^1 g(x)e^{-j\pi i x}\,dx$ we have
$|\mathcal{F}(g)[j]|\le \frac12\|g\|_{L^1}$ and $\|e^{j\pi i x}\|_{L^1([-1,1])}=2$, hence
\[
\|\pi_K g\|_{L^1}
\le \sum_{j=-K}^K |\mathcal{F}(g)[j]|\;\|e^{j\pi i x}\|_{L^1}
\le (2K+1)\|g\|_{L^1}.
\]
Combining the estimates yields the claim.
\end{proof}


\begin{lemma}\label{c61}
For all $i\in\mathbb N$,
\[
\|P_{\sigma,K}^i\|_{L^2\to L^2}
\le \sqrt{2}\,\|N_\sigma\|_{L^1\to L^2}\,(1+\Gamma^{(1)}_{\sigma,K})^{i-1},
\]
where
\[
\|N_\sigma\|_{L^1\to L^2}
=\|\tau_*\rho_\sigma\|_{L^2([-1,1])}
\le \sqrt{\frac{1}{2\sigma\sqrt{\pi}}\coth\!\Bigl(\frac{1}{2\sigma^2}\Bigr)}.
\]
\end{lemma}

\begin{proof}
Set $A:=\pi_K N_\sigma P$. Using $\pi_K^2=\pi_K$, we have $P_{\sigma,K}^i=A^i\pi_K$.
For $g\in L^2([-1,1])$, Cauchy--Schwarz and $\|\pi_K\|_{L^2\to L^2}\le 1$ give
\[
\|\pi_K g\|_{L^1}\le \sqrt{2}\,\|\pi_K g\|_{L^2}\le \sqrt{2}\,\|g\|_{L^2}.
\]
Moreover, for $f\in L^1$,
\[
\|Af\|_{L^1}\le \|\pi_K N_\sigma\|_{L^1\to L^1}\,\|P\|_{L^1\to L^1}\,\|f\|_{L^1}
\le (1+\Gamma^{(1)}_{\sigma,K})\|f\|_{L^1},
\]
and
\[
\|Af\|_{L^2}\le \|\pi_K\|_{L^2\to L^2}\,\|N_\sigma\|_{L^1\to L^2}\,\|P\|_{L^1\to L^1}\,\|f\|_{L^1}
\le \|N_\sigma\|_{L^1\to L^2}\,\|f\|_{L^1}.
\]
Therefore,
\[
\|P_{\sigma,K}^i g\|_{L^2}
=\|A^i\pi_K g\|_{L^2}
\le \sqrt{2}\,\|N_\sigma\|_{L^1\to L^2}\,(1+\Gamma^{(1)}_{\sigma,K})^{i-1}\,\|g\|_{L^2},
\]
which proves the operator bound. The expression for $\|N_\sigma\|_{L^1\to L^2}$ is Young's inequality,
and the coth bound follows from Lemma~\ref{c54} and Remark~\ref{rem:coth-bound}.
\end{proof}

\begin{lemma}\label{c60}
There exist $K,N\in\mathbb N$ such that for all $k>K$,
\[
\|P_{\sigma,k}^N|_{\mathcal U_0}\|_{L^2\to L^2}\le \frac34.
\]
\end{lemma}

\begin{proof}
Choose $N$ so that $\|P_\sigma^N|_{\mathcal U_0}\|_{L^2\to L^2}\le \frac12$
(Proposition~\ref{prop:mixinggaussian}). For $g\in\mathcal U_0$,
\[
\|P_{\sigma,k}^N g\|_{L^2}
\le \|P_\sigma^N g\|_{L^2}+\|(P_\sigma^N-P_{\sigma,k}^N)g\|_{L^2}
\le \frac12\|g\|_{L^2}+\|P_\sigma^N-P_{\sigma,k}^N\|_{L^2\to L^2}\,\|g\|_{L^2}.
\]
Now,
\[
||P_\sigma^N-P_{\sigma,k}^N||_{L^2\to L^2}\leq \sum_{i=0}^{N-1} ||P_\sigma^i||_{L^2\to L^2}||P_\sigma-P_{\sigma,k}||_{L^1\to L^2}||P_{\sigma,k}^{N-i-1}||_{L^1\to L^1}||g||_{L^1}.
\]

We can now bound each term using
Theorem~\ref{thm:tail_estimate} together with the $L^1$ growth bound in Lemma~\ref{lem:L1norm}.
The resulting discretization error is controlled by the Fourier tail of the Gaussian multiplier and
decays exponentially in $k$, so for $k$ large enough one gets
$\|P_\sigma^N-P_{\sigma,k}^N\|_{L^2\to L^2}\le \frac14$, proving the claim.
\end{proof}


\begin{theorem}\label{thm:computability}
Let $f_\sigma$ be the fixed point of $P_\sigma$, and $f_{\sigma,K}$ the fixed point of $P_{\sigma,K}$.
Then for any $\epsilon>0$ there exists $K$ such that
\[
\|f_\sigma-f_{\sigma,K}\|_{L^2}<\epsilon.
\]
\end{theorem}

\begin{proof}
Apply Theorem~\ref{c27} to the pair of operators $P_\sigma$ and $P_{\sigma,K}$, using the uniform contraction
on $\mathcal U_0$ provided by Lemma~\ref{c60}. The perturbation size is measured by the discretization error
\[
\|P_\sigma-P_{\sigma,K}\|_{L^1\to L^2}
\;\;\le\;\;
\|(1-\pi_K)N_\sigma\|_{L^1\to L^2}
\;+\;
\|\pi_K(N_\sigma P-\!N_\sigma P\pi_K)\|_{L^1\to L^2},
\]
which is controlled by the Fourier tail of the Gaussian multiplier and decays exponentially in $K$
Lemma~\ref{thm:tail_estimate}.
The remaining operator factors in Theorem~\ref{c27} are bounded by Lemma~\ref{c61}, which gives
\[
\|P_{\sigma,K}^i\|_{L^2\to L^2}
\le \sqrt{2}\,\|N_\sigma\|_{L^1\to L^2}\,(1+\Gamma^{(1)}_{\sigma,K})^{i-1},
\qquad
\|N_\sigma\|_{L^1\to L^2}
=\|\tau_*\rho_\sigma\|_{L^2([-1,1])}
\]
note that Lemma~\ref{c61} yields the $C_i$ for Theorem~\ref{c27}.  Since $\Gamma^{(1)}_{\sigma,K}\to 0$ exponentially as $K\to\infty$, the error bound furnished by
Theorem~\ref{c27} can be made smaller than $\epsilon$ by taking $K$ sufficiently large, proving the claim.
\end{proof}

\begin{corollary}\label{cor:BirkhoffAvg}
Let $\phi\in L^2(m)$. Then the Birkhoff average
\[
\frac{1}{N}\sum_{n=0}^{N-1}\phi(X_n)
\]
converges almost surely to $\int \phi\,f_\sigma\,dm$. Moreover, for any $\epsilon>0$ there exists $K$ such that
\[
\Bigl|\int \phi\,f_\sigma\,dm-\int \phi\,f_{\sigma,K}\,dm\Bigr|<\epsilon.
\]
\end{corollary}

\begin{proof}
Uniqueness of the stationary measure (Proposition~\ref{prop:unique}) yields ergodicity, hence the almost-sure
convergence follows from the Birkhoff ergodic theorem (Theorem~\ref{c18}). For the approximation error, by
Cauchy--Schwarz,
\[
\Bigl|\int \phi\,(f_\sigma-f_{\sigma,K})\,dm\Bigr|
\le \|\phi\|_{L^2}\,\|f_\sigma-f_{\sigma,K}\|_{L^2},
\]
and the claim follows from Theorem~\ref{thm:computability}.
\end{proof}

\section{Main Result 2: the needed estimates}
\label{chapter6}

We now apply the theory developed in previous sections to compute a rigorous enclosure of the Lyapunov exponent for the noisy dynamical system
\begin{equation} 
X_{n+1} = \tau(T(X_n) + \Omega_\sigma(n)),
\end{equation}
where $\Omega_\sigma(n)$ are i.i.d.~Gaussian random variables with standard deviation $\sigma$, and $\tau$ is the periodic boundary projection.

Our main result will be to estimate the Lyapunov exponent
\begin{align*}
\lambda(\alpha, \beta, \sigma) = \int_{-1}^1 \ln |T'_{\alpha,\beta}(x)| f_\sigma(x) \, dx,
\end{align*}
where $f_\sigma$ is the stationary density of the annealed Perron-Frobenius operator.

\subsection{Integrability of the Observable}

\begin{lemma}\label{c45}
	Since $\ln|T'_{\alpha,\beta}|\in L^1([-1,1])$, $\mu_{\sigma}=f_{\sigma}dm$ and $f_{\sigma}$ is bounded, we have
	\[
	\ln|T'_{\alpha,\beta}|\in L^1(\mu _{\sigma}).
	\]
	Moreover, the constant
	\[
	\Upsilon := \|\ln|T'_{\alpha,\beta}|\|_{L^2([-1,1])} = \sqrt{2} \left( (\ln(\alpha(1+\beta))-(\alpha-1))^2 + (\alpha-1)^2 \right)^{1/2}
	\]
	will be used in later bounds.
\end{lemma}
\begin{proof}
	Using Hölder's inequality:
	\[
	\int_{-1}^{1}\ln|T'_{\alpha,\beta}|\,d\mu_{\sigma} = \int_{-1}^{1}\ln|T'_{\alpha,\beta}| f_{\sigma} \, dm
	\leq \|\ln|T'_{\alpha,\beta}|\|_{L^2} \|f_{\sigma}\|_{L^2}.
	\]
	Since \(f_\sigma\) is bounded, it lies in \(L^2\), and we compute
	\begin{align*}
	\|\ln|T'_{\alpha,\beta}|\|_{L^2}^2 
	&= 2 \int_0^1 \left( \ln(\alpha(1+\beta)) + (\alpha - 1)\ln x \right)^2 dx \\
	&= 2\left( \ln^2(\alpha(1+\beta)) - 2\ln(\alpha(1+\beta))(\alpha - 1) + 2(\alpha - 1)^2 \right) \\
	&= 2\left( (\ln(\alpha(1+\beta)) - (\alpha - 1))^2 + (\alpha - 1)^2 \right),
	\end{align*}
	which gives the expression for \(\Upsilon\) as stated.
\end{proof}

\subsection{Approximating the Lyapunov Exponent}

The choice of a symmetrized trigonometric polynomial approximating $f_\sigma$, as mentioned in Remark \ref{rem:symetrized} is motivated by the fact that this ensures the Lyapunov exponent is a real number.

\begin{definition}
Let $f_{\sigma,k,s}$ be a symmetrized trigonometric polynomial approximating $f_\sigma$, we define
\begin{align*}
\lambda_s(\alpha,\beta,\sigma, k) = \int_{-1}^1 \ln |T'_{\alpha,\beta}(x)| f_{\sigma, k, s}(x) \, dx.
\end{align*}
\end{definition}

\begin{theorem} \label{p4}
Let $f_{\sigma,k,s}$ be a symmetrized trigonometric polynomial approximating $f_\sigma$. Then
\begin{align*}
|\lambda(\alpha,\beta,\sigma) - \lambda_s(\alpha,\beta,\sigma,k)| \leq \Upsilon \|f_\sigma - f_{\sigma,k,s}\|_{L^2},
\end{align*}
where $\Upsilon = \sqrt{2} ((\ln(\alpha(1+\beta))-(\alpha-1))^2 + (\alpha-1)^2)^{1/2}$.
\end{theorem}

\begin{proof}
By Cauchy-Schwarz,
\begin{align*}
\left| \int_{-1}^1 \ln |T'| (f_\sigma - f_{\sigma,k,s}) \, dx \right|
\leq \|\ln |T'|\|_{L^2} \|f_\sigma - f_{\sigma,k,s}\|_{L^2} = \Upsilon \|f_\sigma - f_{\sigma,k,s}\|_{L^2}.
\end{align*}
\end{proof}

\subsection{Fourier Expansion of the Observable}

\begin{lemma}\label{f39}
The Fourier coefficients of $\ln|T'_{\alpha,\beta}|$ on $[-1,1]$ are
\begin{align*}
\mathcal{F}[0] &= \ln(\alpha(1+\beta)) - (\alpha-1), \\
\mathcal{F}[j] &= -\frac{(\alpha-1)}{j\pi} \int_0^{j\pi} \frac{\sin(x)}{x} dx, \quad j \neq 0.
\end{align*}
\end{lemma}

\begin{remark}
Note that to compute the Fourier expansion of $\ln|T'_{\alpha,\beta}|$ we only need to compute rigorously $\int_{0}^{j\pi}\frac{1}{x}\sin(x)\,dx$. 
	The procedure we use to compute $\int_{0}^{j\pi}\frac{1}{x}\sin(x)\,dx$ combines Taylor models and power series expansions, ensuring both rigor and numerical efficiency:

\begin{itemize}
	\item For intervals \([l\pi, (l+1)\pi]\) where \(l \geq 1\), we use high-order Taylor models to approximate the integrand and compute verified bounds for the integral.
	\item Near the singularity at \(x = 0\), we rely on the power series expansion of \(\frac{\sin x}{x}\), integrating term by term to compute the integral over \([0, \pi]\) with controlled truncation error.
\end{itemize}

All calculations are implemented in interval arithmetic, guaranteeing that the computed bounds are mathematically rigorous. The full computation, can be inspected in full detail in notebook
\begin{center}
 \href{https://github.com/orkolorko/PlateauExperiment.jl/blob/main/notebooks/one_experiment.ipynb}{PlateauExperiment.jl/notebooks/one\_experiment.ipynb}.
\end{center}
\end{remark}

\subsection{Computation Strategy}

Let $\widetilde{f}_{\sigma,k,s}$ be the approximation of $f_\sigma$ on $[0,1]$ obtained via Fourier projection after coordinate rescaling.

\begin{theorem} \label{p5}
The approximate Lyapunov exponent is computed as
\begin{align*}
\lambda_s(\alpha,\beta,\sigma,k) = & \; [\ln(\alpha(1+\beta)) - (\alpha -1)] \cdot \mathcal{F}_{[0,1]}(\widetilde{f}_{\sigma,k,s})[0] \\
& + \sum_{j=1}^{k} \left( \frac{2(\alpha-1)}{j\pi} \int_0^{j\pi} \frac{\sin(x)}{x} dx \right) \cdot \mathcal{F}_{[0,1]}(\widetilde{f}_{\sigma,k,s})[j]
\end{align*}
where $\mathcal{F}_{[0,1]}$ is the Fourier transform on $[0,1]$. 
\end{theorem}

\section{Code}

The code used to perform our rigorous computations is publicly available at the following link: \\
\begin{center}
\href{https://github.com/orkolorko/PlateauExperiment.jl}{\texttt{https://github.com/orkolorko/PlateauExperiment.jl}}.
\end{center}

This program enables the certified computation of tight intervals that contain the Lyapunov exponent of the noise-driven dynamical system described in~\eqref{p3}, for any given values of the parameters \(\alpha\), \(\beta\), and \(\sigma\). All computations are performed with rigorous error bounds using interval arithmetic and validated numerics.


\section*{Acknowledgements}
I.N. was partially supported by the Posgraduate Program in Mathematics at UFRJ, CAPES–Finance Code 001, CNPq Projeto Universal No. 404943/2023-3, CAPES–PRINT No. 88881.311616/2018-00, and CAPES–STINT No. 88887.15574\\6/2017-00.

C.L.V. was partially supported by Posgraduate Program in Mathematics at UFRJ, CAPES-Process No. 88882.331108/2019-01.

The authors would like to thank Prof. Yuzuru Sato for introducing them to the fascinating topic of noise-induced phenomena and for generously sharing his deep understanding of their phenomenological behavior.

The authors used ChatGPTv4 to polish the text for spelling, grammar and general style.

\FloatBarrier
\bibliography{nio_unimodal_gauss}{}
\bibliographystyle{hplain}

\end{document}